\documentclass[12pt, reqno]{amsart}
\usepackage{amssymb,amsthm,amsmath}
\usepackage[numbers,sort&compress]{natbib}
\usepackage{amssymb,amsmath}
\usepackage{amsfonts}
\usepackage{mathrsfs}
\usepackage{latexsym}
\usepackage{amssymb}
\usepackage{amsthm}
\usepackage{indentfirst}
\usepackage{colortbl}
\date{
July 16, 2021}

\let\oldsection\section
\renewcommand\section{\setcounter{equation}{0}\oldsection}

\newtheorem{corollary}{Corollary}[section]
\newtheorem{theorem}{Theorem}[section]
\newtheorem{lemma}{Lemma}[section]
\newtheorem{proposition}{Proposition}[section]
\newtheorem{definition}{Definition}[section]

\newtheorem{remark}{Remark}[section]

\allowdisplaybreaks
\begin{document}
\title[$z$-weak solutions to the primitive equations]{Global well-posedness of $z$-weak solutions to the primitive equations without vertical diffusivity}
\thanks{$^*$Corresponding author}
\keywords{Primitive equations, global well-posedness, $z$-weak solutions, without vertical diffusivity.}
\subjclass[2010]{35Q86, 86A05, 86A10.}

\author[J. Li]{JinKai Li}
\address[J. Li]{South China Research Center for Applied Mathematics and Interdisciplinary Studies, School of Mathematical Sciences, South China Normal University, Guangzhou 510631, China}
\email{jklimath@m.scnu.edu.cn; jklimath@gmail.com}

\author[G. Yuan]{GuoZhi Yuan$^*$}
\address[G. Yuan]{South China Research Center for Applied Mathematics and Interdisciplinary Studies, School of Mathematical Sciences, South China Normal University, Guangzhou 510631, China}
\email{shenggaoxii@163.com}

\begin{abstract}
In this paper, we consider the initial boundary value problem in a cylindrical domain to the three dimensional primitive equations with full eddy viscosity in the momentum equations but with only horizontal eddy diffusivity in the temperature equation. Global well-posedness of $z$-weak solution is established for any such initial datum that itself and its vertical derivative belong to $L^2$. This not only extends the results in \cite{Cao5} from the spatially periodic case to general cylindrical domains but also weakens the regularity assumptions on the initial data which are required to be $H^2$ there.
\end{abstract}
\maketitle

\section{Introduction}
The primitive equations, derived from the Navier-Stokes equation by the hydrostatic approximation, play an important role in weather prediction, see the books \cite{Hal,Lew,Maj,Ped,Val,Was,Zen}. Mathematically, the primitive equations can be rigorously justified through taking the small aspect ratio limit to the Navier-Stokes equations defined on a thin domain, see \cite{Aze,Fur,Fur2,Gao,Jli,Jli2,Liu4,Pu}. The systemic mathematical analysis on the primitive equations was initiated by a series of papers \cite{Lio1,Lio2,Lio3} in 1990s, where the global existence of weak solutions was established; however, the uniqueness of weak solution is still open, even for the two-dimensional case, except for some special cases (see \cite{Bre1,Kuk1,Jli3}). Since then, the mathematical analysis on the primitive equations has been being extensively studied. The local well-posedness of strong solutions was established in \cite{Gui}, while the global existence of strong solutions for the two-dimensional case was proved in \cite{Bre2}. For the three-dimensional case, the global well-posedness with arbitrary large initial data was established in \cite{Cao7}, see also \cite{Kob,Kuk}. One can see \cite{Guo,Ju,Kuk1,Jli3,Med} for the well-posedness of the primitive equations with lower regular initial data and \cite{Gig1,Gig2,Hie1,Hie2} for the results in $L^p$ type spaces, see also \cite{Pet} for the solutions with Gevrey regularities. Recently, there are some global well-posed results of strong solutions to the coupled system of the primitive equations to the moisture system, see \cite{Cot,Guo1,Hit1,Hit2}. For results of the compressible primitive equations, one can see \cite{KAZ,WFC1,Liu1,Liu2,Liu3,WFC2}.

Note that all the systems considered in the papers mentioned in the above paragraph are assumed to have full viscosities and full diffusivities. However, due to the presence of strong turbulence mixing in the horizontal direction in the large scale atmosphere, the eddy viscosities and diffusivities in the horizontal and vertical directions are different. In particular, the horizontal viscosities and diffusivities are much stronger than the vertical ones.
Therefore, both physically and mathematically, it is important to consider the primitive equations with anisotropic dissipations. Towards this direction, the results in \cite{Cao1,Cao2,Cao3,Cao4,Cao5,Cao6} show that the anisotropic primitive equations with full or only horizontal viscosities are globally well-posed, as long as one still has either the horizontal or vertical diffusivities, see \cite{Saa} for the existence of time periodic solutions. Notably, as indicated in \cite{Tit1,Tit2,Won}, the inviscid primitive equations, with or without coupling to the temperature equation, may blow up in finite time. The results in \cite{Cao1,Cao2,Cao3,Cao4,Cao5,Cao6,Tit1,Tit2} indicate that the horizontal viscosity has a crucial effect on the global well-posedness of the primitive equations. The effect of the rotation on the life-span of solutions to the inviscid primitive equations was studied in \cite{Gho}.

In this paper, we continue the studies on the anisotropic primitive equations. More precisely, the main purpose of this paper is to establish the global well posed result of the $z$-weak solutions to the system without vertical eddy diffusivity. We focus on the case that with full viscosities but with only horizontal diffusivities, that is the following system:
\begin{equation}
\partial_t v+(v\cdot\nabla_H) v+w\partial_z v+\nabla_H p-\frac{1}{Re_1}\Delta_H v-\frac{1}{Re_2}\partial_z^2 v+fk\times v=0,\label{1}
\end{equation}
\begin{equation}
\partial_z p+T=0,
\label{2}
\end{equation}
\begin{equation}
\nabla_H\cdot v+\partial_z w=0,
\label{3}
\end{equation}
\begin{equation}
\partial_t T+v\cdot\nabla_H T+w\partial_z T-\frac{1}{R_T}\Delta_H T=0,
\label{4}
\end{equation}
here the horizontal velocity $v=(v_1,v_2)$, the vertical velocity $w$, the temperature $T$, and the pressure $p$ are the unknowns,  $f$ is the Coriolis parameter, and $k\times v=(-v_2,v_1)$. $Re_1$ and $Re_2$ are positive constants representing the horizontal and vertical Reynolds numbers, respectively, and $R_T$ is a positive constant. In this paper, we use $\nabla_H =(\partial_{x},\partial_{y})$ and $\Delta_H=\partial_x^2+\partial_y^2$ to denote the horizontal gradient and Laplacian, respectively.

We consider the above system in a cylindrical domain
\begin{equation*}
\Omega=M\times(-h,0),
\end{equation*}
where $M$ is a given smooth bounded domain in $\mathbb{R}^2$ and $h>0$ is a given positive number. The boundary of $\Omega$ consists of the surface $\Gamma_u$, the bottom $\Gamma_b$, and the side $\Gamma_s$:
$$
\Gamma_u=M\times \{0\},\quad
\Gamma_b=M\times \{-h\},\quad
\Gamma_s=\partial M\times (-h,0).
\label{7}
$$

We complement system \eqref{1}--\eqref{4} with the initial condition
\begin{equation}
(v,T)|_{t=0}=(v_0,T_0),\label{8}
\end{equation}
and the following boundary conditions:
\begin{eqnarray}
\textrm{on}~\Gamma_s:&&\partial_n T=-\alpha_TT,~v=0,
\label{9}\\
\textrm{on}~\Gamma_u:&&\partial_z v=0,~w=0,\label{10}\\
\textrm{on}~\Gamma_b:&&\partial_z v=0,~w=0,\label{11}
\end{eqnarray}
where $n$ is the unit outward normal vector on the boundary $\Gamma_s$. For simplicity, we assume that the $\alpha_T$ is a nonnegative constant; however, with the similar arguments as in this paper, the results continue to hold if it varies on the boundary or depends also on the time variable, as long as it is nonnegative and suitably smooth.

\begin{remark}
In general, one can consider the following nonhomogeneous boundary conditions to system \eqref{1}--\eqref{4}:
\begin{equation*}
\begin{aligned}
&\textrm{on}~\Gamma_s:~~~~\partial_n T=\alpha_T(T_{s}-T),~v=0,\\
&\textrm{on}~\Gamma_u:~~~~\partial_z v=-\alpha_v\tau(x,y;t),~w=0,\\
&\textrm{on}~\Gamma_b:~~~~\partial_z v=0,~w=0,
\end{aligned}
\end{equation*}
where $\alpha_v$ is a nonnegative constant or a nonnegative suitably smooth function on the boundary $\Gamma_u$, $\tau(x,y,t)$ is the wind stress on the ocean surface, and $T_{s}(x,y,z,t)$ is the typical temperature on $\Gamma_s$. Assume that $\tau$ and $T_{s}$ are sufficiently smooth, uniformly bounded, and satisfy the following compatibility conditions:
\begin{equation}
\begin{aligned}
&\tau=0~\textrm{on}~\Gamma_s,~~\textrm{if}~\alpha_v>0, \\
&\partial_zT_s=0~\textrm{on}~\partial M\times \{-h\}~\textrm{and}~\partial M\times \{0\},~~\textrm{if}~\alpha_T>0.\label{12}
\end{aligned}
\end{equation}
As will be shown in the Appendix, due to \eqref{12}, one can homogenize the above nonhomogeneous boundary conditions to the corresponding homogeneous ones at the cost of dealing with some additional linear terms in the resulting equations. These additional linear term will bring no technical difficulties if we assume that both $\tau$ and $T_s$ are sufficiently smooth. Therefore, even though we only focus on the case that $\tau=0$ and $T_s=0$ in this paper, the same results still hold in the general case that $\tau\neq0$ and $T_s\neq0$, as long as they are suitably smooth.
\end{remark}

By \eqref{3} and the boundary conditions of $w$ in \eqref{10} and \eqref{11}, one gets
\begin{equation*}
\nabla_H\cdot\left(\int_{-h}^0v dz\right)=0 
\end{equation*}
and
\begin{equation}
w(x,y,z,t)=-\int_{-h}^z\nabla_H\cdot v(x,y,\xi,t)d\xi.\label{14}
\end{equation}
Integrating equation \eqref{2} with respect to $z$ yields
\begin{equation}
p(x,y,z,t)=-\int_{-h}^z T(x,y,\xi,t)d\xi+p_s(x,y,t) \label{15}
\end{equation}
for some unknown function $p_s(x,y,t)$.

Based on the above, one can rewrite system \eqref{1}--\eqref{4} as:
\begin{equation}
\begin{aligned}
\partial_t v+(v\cdot\nabla_H) v&-\left(\int_{-h}^z\nabla_H\cdot vd\xi\right)\partial_z v+\nabla_H p_s(x,y,t)\\
&-\nabla_H\left(\int_{-h}^zTd\xi\right)+fk\times v-\frac{1}{Re_1}\Delta_H v-\frac{1}{Re_2}\partial_z^2 v=0,\label{16}
\end{aligned}
\end{equation}
\begin{equation}
\nabla_H\cdot\left(\int_{-h}^0v dz\right)=0,\label{13}
\end{equation}
\begin{equation}
\partial_t T+v\cdot\nabla_H T-\left(\int_{-h}^z\nabla_H\cdot vd\xi\right)\partial_z T-\frac{1}{R_T}\Delta_H T=0,\label{17}
\end{equation}
subject to the following boundary conditions
\begin{equation}
\begin{aligned}
\textrm{on}~\Gamma_s:~~~~\partial_n T=-\alpha_TT,~v=0,
\label{73}
\end{aligned}
\end{equation}
\begin{equation}
\textrm{on}~\Gamma_u\cup\Gamma_b:~~~~\partial_z v=0.
\label{75}
\end{equation}

Now, we define the following two spaces $\mathcal H$ and $\mathcal V$ which will be used throughout this paper:
\begin{equation*}
\mathcal H:=\left\{\varphi\in L^2(\Omega)~\Bigg|~n\cdot\left(\int_{-h}^0\varphi dz\right)=0\mbox{ on }\partial M,  \nabla_H\cdot\left(\int_{-h}^0 \varphi dz\right)=0\ \  \textrm{in}~\Omega\right\},
\end{equation*}
\begin{equation*}
\mathcal V:=\left\{\varphi\in H^1(\Omega)~\Bigg|~\varphi=0~\textrm{on}~\Gamma_s,~\nabla_H\cdot\left(\int_{-h}^0\varphi dz\right)=0~\textrm{in}~\Omega\right\}.
\end{equation*}
As usual, we use $L^q(\Omega)$ and $W^{m,q}(\Omega)$ to denote the standard Lebesgue and Sobolev spaces, respectively. For $q=2$, we use $H^m$ instead of $W^{m,2}$, and we still use $L^q$ and $H^m$ to denote the $N$ product spaces $(L^q)^N$ and $(H^m)^N$, respectively. If without confusion, we will use $\|\cdot\|_{L^q}$ to denote the $L^q$ norm and use $\|(f_1,...,f_n)\|_{L^2}^2$ to denote the summation $\sum_{i=1}^n\|f_i\|_{L^2(\Omega)}^2$. For simplicity, we use $d\Omega$ instead of $dxdydz$ sometimes.

We start with the definition of $z$-weak solutions to system \eqref{16}--\eqref{17}, subject to the boundary conditions \eqref{73}--\eqref{75} and the initial condition (\ref{8}).

\begin{definition}
\label{def}
Given $(v_0,T_0)\in \mathcal H\times L^2(\Omega)$ with $(\partial_zv_0,\partial_zT_0)\in L^2(\Omega)$.
A pair $(v,T)$ is called a global $z$-weak solution to system \eqref{16}--\eqref{17}, subject to the boundary conditions \eqref{73}--\eqref{75} and the initial condition (\ref{8}), if it has the regularities
\begin{equation*}
\begin{aligned}
&(v,\partial_zv)\in L^\infty(0,\mathfrak{T};L^2(\Omega))\cap L^2(0,\mathfrak{T};H^1(\Omega)),\\
&(T,\partial_zT)\in L^\infty(0,\mathfrak{T};L^2(\Omega)),\quad (\nabla_H T,\nabla_H\partial_zT)\in L^2(\Omega\times(0,\mathfrak{T})),\\
&\sqrt{t}(\nabla v,\nabla_HT)\in L^\infty(0,\mathfrak{T};L^2(\Omega))\cap L^2(0,\mathfrak{T};H^1(\Omega)),\\
&\sqrt{t}(\partial_tv, \partial_tT)\in L^2(\Omega\times(0,\mathfrak{T})),\quad (v,T)\in C([0,\mathfrak T];\mathcal H\times L^2(\Omega)),
\end{aligned}
\end{equation*}
for any positive time $\mathfrak{T}$, satisfies equations \eqref{16}--\eqref{17} a.e.\,in $\Omega\times(0,\infty)$, and fulfills the boundary conditions \eqref{73}--\eqref{75} and the initial condition (\ref{8}).
\end{definition}

We are now ready to state the main result of this paper.

\begin{theorem}\label{thm1}
Let $(v_0,T_0)\in\mathcal H\times L^2(\Omega)$ with $(\partial_z v_0,\partial_zT_0)\in L^2(\Omega)$. Then, there is a unique global $z$-weak solution $(v,T)$ to system \eqref{16}--\eqref{17}, subject to the boundary conditions \eqref{73}--\eqref{75} and the initial condition (\ref{8}). Moreover, the unique solution is continuously depending on the initial data.
\end{theorem}

\begin{remark}
Global well-posedness of strong solutions to the periodic boundary value problem for system \eqref{16}--\eqref{17} has
already proved in \cite{Cao5}, where the initial datum is assumed to be $H^2$. Note that the regularities required in Theorem \ref{thm1} on the initial data are much weaker than those in \cite{Cao5}. Therefore, Theorem \ref{thm1} not only extends the results in \cite{Cao5} from the spatially periodic case to general cylindrical domains but also weakens the regularity assumptions on the initial data. The
arguments presented in this paper work also for the periodic boundary value problem.
\end{remark}

\begin{remark}
The same result still holds if the horizontal velocity $v$ satisfies the slip boundary conditions rather than the non-slip boundary conditions on $\Gamma_s$. In other words, one still has the global existence and uniqueness of $z$-weak solutions to the same system
if replacing the boundary condition $v=0$ on $\Gamma_s$ by the following
\begin{equation*}
v\cdot\overrightarrow{n}=0,\ \  \partial_n v\times\overrightarrow{n}=0,
\end{equation*}
and at the same time changing the compatibility condition of $\tau$ in \eqref{12} to 
\begin{equation*}
\tau\cdot\overrightarrow{n}=0,\ \  \partial_n \tau\times\overrightarrow{n}=0,\ \ \ \ on~\Gamma_s.
\end{equation*}
In fact, the arguments presented in this paper still work by slightly modifying the calculations.
\end{remark}

The rest of this paper is arranged as follows: in Section 2, some preliminary results being used in the subsequent sections are collected, including the local existence result to a regularized system with full dissipation; in Section 3, global well-posedness and a priori estimates depending on $\|(v_0, T_0)\|_{H^1}$
are established for the regularized system with $H^1$ initial datum; this is the base to construct approximating sequence to $z$-weak solutions in section 5; in section 4,
global in time a priori estimate depending only on $\|(v_0, T_0, \partial_zv_0,\partial_zT_0)\|_{L^2}$ are carried out;
Theorem \ref{thm1} is proved in section 5; in the Appendix, we give some details about how to transform the problem with nonhomogeneous boundary conditions to the corresponding problem with homogeneous boundary conditions.

Throughout this paper we use $C$ to denote a generic positive constant which may vary from line to line. For simplicity
of presentations, the dependence of $C$ on the parameters or quantities is stated only in the statements of the
theorems, propositions, corollaries, or lemmas, rather in their proofs.

\section{Preliminaries}
In this section, we collect some preliminary results which will be used in the rest of this paper.

The following inequality will be used frequently in the a priori estimates, one can prove it in the same way as in \cite{Cao8} and \cite{Cao1}, and thus we omit its proof here.

\begin{lemma}\label{lemma1}
Let $S$ be a bounded domain in $\mathbb{R}^2$ and denote by $L$ the diameter of $S$. Then, the following inequalities hold:
\begin{equation*}
\begin{aligned}
&\ \ \  \int_S\left( \int_{-h}^0|\phi(x,y,z)|dz\right)\left(\int_{-h}^0|\varphi(x,y,z)\psi(x,y,z)|dz\right)dxdy\\
&\leq C\|\phi\|_{L^2}\|\varphi\|_{L^2}^{\frac{1}{2}}\left(\frac{\|\varphi\|_{L^2}}{L}+\|\nabla_H\varphi\|_{L^2}\right)^{\frac{1}{2}}\|\psi\|_{L^2}^{\frac{1}{2}}\left(\frac{\|\psi\|_{L^2}}{L}+\|\nabla_H\psi\|_{L^2}\right)^{\frac{1}{2}}
\end{aligned}
\end{equation*}
and
\begin{equation*}
\begin{aligned}
&\ \ \  \int_S\left( \int_{-h}^0|\phi(x,y,z)|dz\right)\left(\int_{-h}^0|\varphi(x,y,z)\psi(x,y,z)|dz\right)dxdy\\
&\leq C\|\psi\|_{L^2}\|\varphi\|_{L^2}^{\frac{1}{2}}\left(\frac{\|\varphi\|_{L^2}}{L}+\|\nabla_H\varphi\|_{L^2}\right)^{\frac{1}{2}}\|\phi\|_{L^2}^{\frac{1}{2}}\left(\frac{\|\phi\|_{L^2}}{L}+\|\nabla_H\phi\|_{L^2}\right)^{\frac{1}{2}},
\end{aligned}
\end{equation*}
here we denote $\|\cdot\|_{L^q}=\|\cdot\|_{L^q(S\times (-h,0))}$, for any $\phi$, $\varphi$ and $\psi$, such that the quantities on the right hand sides are finite, and $C$ is a constant depending only on the shape of $S$, but not on its size.
\end{lemma}

We also need the following Aubin--Lions lemma to obtain some compactness of the bounded sequences.

\begin{lemma}[see Corollary 4 of \cite{Sim}]\label{lemma2}
Let $t^*\in (0,\infty)$ be given. Assume that $X$, $Y$ and $Z$ are three Banach spaces such that $X\hookrightarrow\hookrightarrow Y\hookrightarrow Z$. Then, it holds that

(i)\ \  if $F$ is a bounded subset of $L^p(0,t^*;X)$ where $1\leq p<\infty$, and $\frac{\partial F}{\partial t}=\{\frac{\partial f}{\partial t}|f\in F\}$ is bounded in $L^1(0,t^*;Z)$, then $F$ is relatively compact in $L^p(0,t^*; Y)$;

(ii)\ \ if $F$ is a bounded subset of $L^\infty(0,t^*;X)$ , and $\frac{\partial F}{\partial t}$ is bounded in $L^r(0,t^*;Z)$, where $r>1$, then $F$ is relatively compact in $C([0,t^*];Y)$.
\end{lemma}

The following lemma is taken from \cite{Temam} (see Lemma 1.4 on page 263).

\begin{lemma}\label{lemmaADLI}
Let $X$ and $Y$ be two Banach spaces, such that $X\hookrightarrow Y$.
If a function $\phi$ belongs to $L^\infty(0, \mathfrak T; X)$ and is weakly continuous with
values in $Y$, then $\phi$ is weakly continuous with values in $X$.
\end{lemma}

Now we consider the system with full dissipation. By adding a vertical diffusivity term $-\varepsilon\partial_z^2T$ with a small parameter $\varepsilon\in(0,1)$ in the temperature equation, we have the following regularized system with full dissipation:
\begin{equation}
\begin{aligned}
\partial_t v-\frac{1}{Re_1}\Delta_H v&-\frac{1}{Re_2}\partial_z^2 v+(v\cdot\nabla_H)v-\left(\int_{-h}^z\nabla_H\cdot vd\xi\right)\partial_z v\\
&+\nabla_H p_s-\nabla_H\int_{-h}^zTd\xi+fk\times v=0,\label{19}
\end{aligned}
\end{equation}
\begin{equation}
\nabla_H\cdot\int_{-h}^0vdz=0,\label{20}
\end{equation}
\begin{equation}
\partial_tT-\frac{1}{R_T}\Delta_HT-\varepsilon\partial_z^2T+v\cdot\nabla_HT-\left(\int_{-h}^z\nabla_H\cdot vd\xi\right)\partial_z T=0.\label{21}
\end{equation}
We complement the above system with the initial condition 
\begin{equation}
\begin{aligned}
\big(v(t),T(t)\big)\big|_{t=0}&=\big(v_0,T_0\big)\label{22}
\end{aligned}
\end{equation}
and the boundary conditions
\begin{equation}
\begin{aligned}
\textrm{on}~\Gamma_s:\ \  &v=0,~~\partial_nT+\alpha_TT=0,\label{23}
\end{aligned}
\end{equation}
\begin{equation}
\textrm{on}~\Gamma_u~\textrm{and}~\Gamma_b:\  \partial_z v=0,~~\partial_zT=0.\label{24}
\end{equation}

Note that here we impose the homogeneous Neumann type rather than the Dirichlet type boundary conditions on $\Gamma_u$ and $\Gamma_b$ for $T$. The reason is that we will eventually pass the vertical diffusivity $\varepsilon$ to zero, while the Neumann boundary conditions will not produce the boundary layer during this limit procedure.

Finally, for any fixed $\varepsilon>0$, the following short time existence result holds 
for the above regularized system, which can be proved in the same way as in \cite{Gui}.

\begin{proposition}\label{pro2.1}
Let $(v_0,T_0)\in \mathcal V\times H^1(\Omega)$. Then, for any $\varepsilon\in(0,1)$, there is a positive time $t_\varepsilon>0$ such that system \eqref{19}--\eqref{21}, subject to \eqref{22}--\eqref{24}, has a unique solution $(v,T)$, on $\Omega\times(0,t_\varepsilon)$, satisfying
\begin{equation*}
(v,T)\in C([0,t_\varepsilon];H^1(\Omega))\cap L^2(0,t_\varepsilon;H^2(\Omega)),~~(\partial_tv,\partial_tT)\in L^2((0,t_\varepsilon)\times \Omega).
\end{equation*}
\end{proposition}


\section{Global strong solutions with $H^1$ initial data}
\label{secstrong}
This section is denoted to establishing the global well-posedness of strong solutions to the regularized system
\eqref{19}--\eqref{21}, subject to \eqref{22}--\eqref{24}, and deriving the $\varepsilon$--independent a priori estimates. 

Given $(v_0,T_0)\in\mathcal V\times H^1(\Omega)$. By Proposition \ref{pro2.1}, there is a unique local solution $(v,T)$ to system \eqref{19}--\eqref{21}, subject to \eqref{22}--\eqref{24}. By applying Proposition \ref{pro2.1} iteratively, one can extend $(v,T)$ to the maximal time of existence $t^*_\varepsilon$. Due to this, throughout this section until the last subsection, we always assume that $(v,T)$ has been extended to the maximal time of existence $t^*_\varepsilon$.

The a priori estimates for $(v,T)$ are carried out in the following order: the low-order energy estimates, estimates of $(\partial_zv,\partial_zT)$, and estimates of $(\nabla_Hv,\nabla_HT)$ and $(\partial_tv,\partial_tT)$. There are carried out separately in the following three subsections.

\subsection{Low-order energy estimates}

First, one can obtain the following basic energy estimate.
\begin{proposition}\label{pro3.1}
For any positive time $\mathfrak{T}\in(0,t^*_\varepsilon)$, the following holds
\begin{align*}
&\ \  \sup_{0\leq t\leq \mathfrak{T}}\|(v,T)\|_{L^2}^2(t)+\frac{1}{R_T}\|\nabla_HT\|_{L^2(\Omega\times(0,\mathfrak{T}))}^2
+\frac{1}{Re_1}\|\nabla_Hv\|_{L^2(\Omega\times(0,\mathfrak{T}))}^2 \\
&\qquad\qquad
 +\frac{1}{Re_2}\|\partial_zv\|_{L^2(\Omega\times(0,\mathfrak{T}))}^2\leq e^{C\mathfrak{T}}\|(v_0,T_0)\|_{L^2}^2,\label{31}
\end{align*}
where $C$ is a positive constant depending only on $h$ and $R_T$.
\end{proposition}

\begin{proof}
Multiplying \eqref{19} and (\ref{21}), respectively, with $v$ and $T$, and integrating over $\Omega$, it follows from integrating by parts and
\eqref{20} that
\begin{equation*}
\begin{aligned}
&\ \ \frac{1}{2}\frac{d}{dt}\|v\|_{L^2}^2+\frac{1}{Re_1}\|\nabla_H v\|_{L^2}^2+\frac{1}{Re_2}\|\partial_z v\|_{L^2}^2\\
&=\left(\nabla_H\int_{-h}^zTd\xi,v\right)_{L^2}
\leq C\|\nabla_H T\|_{L^2}\|v\|_{L^2} \label{29}
\end{aligned}
\end{equation*}
and 
\begin{equation*}
\frac{1}{2}\frac{d}{dt}\|T\|_{L^2}^2+\frac{1}{R_T}\|\nabla_HT\|_{L^2}^2+\frac{\alpha_T}{R_T}\|T\|_{L^2(\Gamma_s)}^2+\varepsilon\|\partial_zT\|_{L^2}^2
=0,\label{30}
\end{equation*}
here and what follows, we use $(\cdot,\cdot)_{L^2}$ to denote the $L^2(\Omega)$ inner product.
Combining the above two and using the Young inequality lead to
\begin{equation*}
\begin{aligned}
&\frac{d}{dt}\big(\|v\|_{L^2}^2+\|T\|_{L^2}^2\big)+\frac{1}{R_T}\|\nabla_HT\|_{L^2}^2+\varepsilon\|\partial_zT\|_{L^2}^2\\
&\qquad\qquad+\frac{\alpha_v}{R_T}\|T\|_{L^2(\Gamma_s)}^2+\frac{1}{Re_1}\|\nabla_Hv\|_{L^2}^2+\frac{1}{Re_2}\|\partial_zv\|_{L^2}^2
\leq C\|v\|_{L^2}^2,
\end{aligned}
\end{equation*}
from which, by the Gr\"onwall inequality, the conclusion follows.
\end{proof}

Decompose
\begin{equation*}
v=\widetilde{v}+\overline{v},
\end{equation*}
where
\begin{equation*}
\overline{v}:=\frac{1}{h}\int_{-h}^0vdz\qquad\textrm{and}\qquad\widetilde{v}:=v-\overline{v}.
\end{equation*}
One can check that $\overline{v}$ satisfies
\begin{equation}
\begin{aligned}
&\ \ \ \partial_t\overline{v}+\nabla_Hp_s-\frac{1}{Re_1}\Delta_H\overline{v}\\
&=-(\overline{v}\cdot\nabla_H)\overline{v}-\frac{1}{h}\int_{-h}^0\Big((\widetilde{v}\cdot\nabla_H)\widetilde{v}+(\nabla_H\cdot\widetilde{v})\widetilde{v}\Big)dz\\
&\ \ \ +\frac{1}{h}\int_{-h}^0\left(\nabla_H\int_{-h}^zTd\xi\right)dz-fk\times\overline{v},\ \ \ \qquad ~\textrm{in}~M,
\end{aligned}\label{32}
\end{equation}
$\nabla_H\cdot\overline{v}=0$ in $M$,
and the following boundary conditions
\begin{equation}
\begin{aligned}
&\overline{v}=0\ \ \ \textrm{on}~\partial M.\label{33}
\end{aligned}
\end{equation}
For $\widetilde{v}$, one can check that it satisfies
\begin{equation}
\begin{aligned}
&\ \ \ \partial_t\widetilde{v}-\frac{1}{Re_1}\Delta_H\widetilde{v}-\frac{1}{Re_2}\partial_z^2\widetilde{v}\\
&=-(\widetilde{v}\cdot\nabla_H)\widetilde{v}-(\widetilde{v}\cdot\nabla_H)\overline{v}-(\overline{v}\cdot\nabla_H)\widetilde{v}+\left(\int_{-h}^z\nabla_H\cdot\widetilde{v}d\xi\right)\partial_z\widetilde{v}\\
&\ \ \ +\nabla_H\int_{-h}^zTd\xi-\frac{1}{h}\int_{-h}^0\left(\nabla_H\int_{-h}^zTd\xi\right) dz-fk\times\widetilde{v}\\
&\ \ \ +\frac{1}{h}\int_{-h}^0\Big((\widetilde{v}\cdot\nabla_H)\widetilde{v}+(\nabla_H\cdot\widetilde{v})\widetilde{v}\Big)dz,\ \ \  ~\qquad \textrm{in}~\Omega,\label{34}
\end{aligned}
\end{equation}
and the following boundary conditions
\begin{equation}
\begin{aligned}
&\widetilde{v}=0\ \ \ \ \ \ \ \ \ \ \ \ \textrm{on}~\Gamma_s,\\
&\partial_z\widetilde{v}=0\ \ \ \ \ \ \ \ \  \textrm{on}~\Gamma_b,\\
&\partial_z\widetilde{v}=0\ \ \ \ \ \ \ \ \  \textrm{on}~\Gamma_u.\label{35}
\end{aligned}
\end{equation}

The next proposition states the $L^\infty(0,\mathfrak{T};L^{3+\delta}(\Omega))$ estimate for $\widetilde{v}$ with $\delta\in(0,1]$.
\begin{proposition}\label{pro3.2}
For any positive time $\mathfrak{T}\in(0,t^*_\varepsilon)$ and for any positive number $\delta\in(0,1]$, the following holds
$$
\sup_{0\leq t\leq\mathfrak{T}}\|\widetilde{v}\|_{L^{3+\delta}}^{3+\delta}(t) \leq e^{C\int_{0}^{\mathfrak{T}}K_1(t)dt}\big(\|\widetilde{v}_0\|_{L^{3+\delta}}^{3+\delta}+1\big)
$$
for a positive constant $C$ depending only on $\delta, h$, and $Re_1$, where
\begin{equation*}
K_1(t):=\|\nabla_H v\|_{L^2}^2\big(\|v\|_{L^2}^2+1\big)+\big(\|T\|_{L^2}^2+1\big)\big(\|\nabla_H T\|_{L^2}+1\big).
\end{equation*}
\end{proposition}
\begin{proof}
Multiplying \eqref{34} with $|\widetilde{v}|^{1+\delta}\widetilde{v}$ and integrating the resulting over $\Omega$, it follows from integrating by parts and the boundary conditions \eqref{35} that
\begin{align*}
\frac{1}{3+\delta}\frac{d}{dt}\|\widetilde{v}\|_{L^{3+\delta}}^{3+\delta}&+\frac{1}{Re_1}
\int_\Omega[|\nabla_H\widetilde{v}|^2|\widetilde{v}|^{1+\delta}+(1+\delta)
\big|\nabla_H|\widetilde{v}|\big|^2|\widetilde{v}|^{1+\delta}]d\Omega\\
&+\frac{1}{Re_2}\int_\Omega[|\partial_z\widetilde{v}|^2|\widetilde{v}|^{1+\delta}+(1+\delta) \big|\partial_z|\widetilde{v}|\big|^2|\widetilde{v}|^{1+\delta}]d\Omega\\
=&-\int_\Omega (\widetilde{v}\cdot\nabla_H)\overline{v}\cdot|\widetilde{v}|^{1+\delta}\widetilde{v}~d\Omega-\int_\Omega(\overline{v}\cdot\nabla_H)\widetilde{v}\cdot|\widetilde{v}|^{1+\delta}\widetilde{v}~d\Omega\\
&+\frac{1}{h}\int_\Omega\left[\int_{-h}^0\Big((\widetilde{v}\cdot\nabla_H)\widetilde{v}+(\nabla_H\cdot\widetilde{v})\widetilde{v}\Big)dz\right]\cdot|\widetilde{v}|^{1+\delta}\widetilde{v}~d\Omega\\
&-\frac{1}{h}\int_\Omega\left[\int_{-h}^0\left(\nabla_H\int_{-h}^zTd\xi\right) dz\right]\cdot|\widetilde{v}|^{1+\delta}\widetilde{v}~d\Omega\\
&+\int_\Omega\left(\nabla_H\int_{-h}^zTd\xi\right)\cdot|\widetilde{v}|^{1+\delta}\widetilde{v}~d\Omega
=:\sum_{i=1}^5 I_i.
\end{align*}
By integrating by parts and using the H\"older, Ladyzhenskaya, and Minkowski inequalities, one deduces
\begin{equation*}
\begin{aligned}
& |I_1|+|I_2|
\leq C_\delta\int_\Omega |\overline{v}||\nabla_H\widetilde{v}||\widetilde{v}|^{2+\delta}d\Omega\\
\leq& C_\delta\int_M |\overline{v}|\left(\int_{-h}^0|\nabla_H\widetilde{v}|^2|\widetilde{v}|^{1+\delta}dz\right)^{\frac{1}{2}}\left(\int_{-h}^0|\widetilde{v}|^{3+\delta}dz\right)^{\frac{1}{2}}dxdy\\
\leq& C_\delta \left(\int_\Omega|\nabla_H\widetilde{v}|^2|\widetilde{v}|^{1+\delta}d\Omega\right)^{\frac{1}{2}}\left(\int_M|\overline{v}|^4 dxdy\right)^{\frac{1}{4}}\left[\int_M\left(\int_{-h}^0|\widetilde{v}|^{3+\delta}dz\right)^2dxdy\right]^{\frac{1}{4}}\\
\leq& C_\delta \left(\int_\Omega|\nabla_H\widetilde{v}|^2|\widetilde{v}|^{1+\delta}d\Omega\right)^{\frac{1}{2}}
\|\overline{v}\|_{L^2(M)}^{\frac{1}{2}}\|\nabla_H\overline{v}\|_{L^2(M)}^{\frac{1}{2}}\\
&\times\left[\int_{-h}^0\left(\int_M|\widetilde{v}|^{6+2\delta}dxdy\right)^{\frac{1}{2}}dz\right]^{\frac{1}{2}},\\
\end{aligned}
\end{equation*}
where the Poincar\'e inequality guaranteed by the boundary condition $\overline{v}|_{\partial M}=0$ has been used. By the Ladyzhenskaya inequality, it has
\begin{align*}
&\ \ \ \left[\int_{-h}^0\left(\int_M|\widetilde{v}|^{6+2\delta}dxdy\right)^{\frac{1}{2}}dz\right]^{\frac{1}{2}}
=\left(\int_{-h}^0\Big\||\widetilde{v}|^{\frac{3}{2}+\frac{\delta}{2}}\Big\|_{L^4(M)}^2dz\right)^{\frac{1}{2}}\\
&\leq C_\delta\left[\int_{-h}^0\left(\int_M|\widetilde{v}|^{3+\delta}dxdy\right)^{\frac{1}{2}}\left(\int_M\big|\nabla_H|\widetilde{v}|\big|^2|\widetilde{v}|^{1+\delta}dxdy\right)^{\frac{1}{2}}dz\right]^{\frac{1}{2}}\\
&\leq C_\delta\|\widetilde{v}\|_{L^{3+\delta}}^{\frac{3+\delta}{4}}\left(\int_\Omega\big|\nabla_H|\widetilde{v}|\big|^2|\widetilde{v}|^{1+\delta}d\Omega\right)^{\frac{1}{4}},
\end{align*}
where again the Poincar\'e inequality has been used. Therefore,
\begin{equation*}
\begin{aligned}
|I_1|+|I_2|\leq C_\delta \left(\int_\Omega|\nabla_H\widetilde{v}|^2|\widetilde{v}|^{1+\delta}d\Omega\right)^{\frac{3}{4}}\|\overline{v}\|_{L^2(M)}^{\frac{1}{2}}\|\nabla_H\overline{v}\|_{L^2(M)}^{\frac{1}{2}}\|\widetilde{v}\|_{L^{3+\delta}}^{\frac{3+\delta}{4}}.
\end{aligned}
\end{equation*}
For $I_3$, integrating by parts and the H\"older inequality yield
\begin{align*}
|I_3|&= \frac{1}{h}\Bigg|\int_\Omega\left(\int_{-h}^0\widetilde{v}\otimes\widetilde{v}dz\right):\nabla_H(|\widetilde{v}|^{1+\delta}\widetilde{v})d\Omega\Bigg|\\
&\leq C_\delta\int_M\left(\int_{-h}^0|\widetilde{v}|^2dz\right)\left(\int_{-h}^0|\widetilde{v}|^{1+\delta}|\nabla_H\widetilde{v}|dz\right)dxdy\\
&\leq C_\delta\int_M\left(\int_{-h}^0|\widetilde{v}|^2dz\right)\left(\int_{-h}^0|\nabla_H\widetilde{v}|^2|\widetilde{v}|^{1+\delta}dz\right)^{\frac{1}{2}}\left(\int_{-h}^0|\widetilde{v}|^{1+\delta}dz\right)^{\frac{1}{2}}dxdy\\
&\leq C_\delta\left[\int_M\left(\int_{-h}^0|\widetilde{v}|^2dz\right)^3dxdy\right]^{\frac{1}{3}}\left(\int_\Omega|\nabla_H\widetilde{v}|^2|\widetilde{v}|^{1+\delta}d\Omega\right)^{\frac{1}{2}}\\
&\qquad\times\left[\int_M\left(\int_{-h}^0|\widetilde{v}|^{1+\delta}dz\right)^3dxdy\right]^{\frac{1}{6}}.
\end{align*}
By the Minkowski, Gagliardo--Nirenberg, and Poincar\'e inequalities, one deduces
\begin{align*}
&\left[\int_M\left(\int_{-h}^0|\widetilde{v}|^2dz\right)^3dxdy\right]^{\frac{1}{3}}\leq \int_{-h}^0\left(\int_M|\widetilde{v}|^6dxdy\right)^{\frac{1}{3}}dz\\
\leq& \int_{-h}^0\|\nabla_H \widetilde{v}\|_{L^2(M)}^{\frac{3-\delta}{3}}\|\widetilde{v}\|_{L^{3+\delta}(M)}^{\frac{3+\delta}{3}}dz\leq C_\delta\|\nabla_H\widetilde{v}\|_{L^2(\Omega)}^{\frac{3-\delta}{3}}\|\widetilde{v}\|_{L^{3+\delta}}^{\frac{3+\delta}{3}} 
\end{align*}
and
\begin{eqnarray*}
\left[\int_M\left(\int_{-h}^0|\widetilde{v}|^{1+\delta}dz\right)^3dxdy\right]^{\frac{1}{6}}\leq \left[\int_{-h}^0\left(\int_M|\widetilde{v}|^{3+3\delta}dxdy\right)^{\frac{1}{3}}dz\right]^{\frac{1}{2}}\\
\leq \left(\int_{-h}^0\|\nabla_H\widetilde{v}\|_{L^2(M)}^{\frac{2\delta}{3}}\|\widetilde{v}\|_{L^{3+\delta}(M)}^{\frac{3+\delta}{3}}dz\right)^{\frac{1}{2}}
\leq C_\delta\|\nabla_H\widetilde{v}\|_{L^2}^{\frac{\delta}{3}}\|\widetilde{v}\|_{L^{3+\delta}}^{\frac{3+\delta}{6}},
\end{eqnarray*}
where $0<\delta\leq1$ has been used. Therefore,
\begin{equation*}
\begin{aligned}
|I_3|\leq C_\delta\left(\int_\Omega|\nabla_H\widetilde{v}|^2|\widetilde{v}|^{1+\delta}d\Omega\right)^{\frac{1}{2}}
\|\nabla_H\widetilde{v}\|_{L^2}\|\widetilde{v}\|_{L^{3+\delta}}^{\frac{3+\delta}{2}}.
\end{aligned}
\end{equation*}
Integrating by parts and using the H\"older, Minkowski and Ladyzhenskaya inequalities, one deduces
\begin{align*}
&|I_4|+|I_5|\\
\leq & C_\delta\int_M\left(\int_{-h}^0|T|dz\right)\left(\int_{-h}^0|\nabla_H\widetilde{v}||\widetilde{v}|^{1+\delta}dz\right)dxdy\\
\leq& C_\delta \int_M\left(\int_{-h}^0|T|dz\right)\left(\int_{-h}^0|\nabla_H\widetilde{v}|^2|\widetilde{v}|^{1+\delta}dz\right)^{\frac{1}{2}}\left(\int_{-h}^0|\widetilde{v}|^{1+\delta}dz\right)^{\frac{1}{2}}dxdy\\
\leq& C_\delta\left(\int_\Omega|\nabla_H\widetilde{v}|^2|\widetilde{v}|^{1+\delta}d\Omega\right)^{\frac{1}{2}}\left[\int_M\left(\int_{-h}^0|T|dz\right)^4dxdy\right]^{\frac{1}{4}}\\
&\times\left[\int_M\left(\int_{-h}^0|\widetilde{v}|^{1+\delta}dz\right)^2dxdy\right]^{\frac{1}{4}}\\
\leq & C_\delta\left(\int_\Omega|\nabla_H\widetilde{v}|^2|\widetilde{v}|^{1+\delta}d\Omega\right)^{\frac{1}{2}}\int_{-h}^0\left(\int_M|T|^4 dxdy\right)^{\frac{1}{4}}dz \\ &\times\left[\int_{-h}^0\left(\int_M|\widetilde{v}|^{2+2\delta}dxdy\right)^{\frac{1}{2}}dz\right]^{\frac{1}{2}}\\
\leq& C_\delta\left(\int_\Omega|\nabla_H\widetilde{v}|^2|\widetilde{v}|^{1+\delta}
d\Omega\right)^{\frac{1}{2}}\|T\|_{L^2(\Omega)}^{\frac{1}{2}}\|(T,\nabla_HT)\|_{L^2}^{\frac{1}{2}}
\|\widetilde{v}\|_{L^{3+\delta}}^{\frac{1+\delta}{2}},
\end{align*}
where $\delta\in(0,1]$ has been used.
Combining the above estimates and applying the Young inequality yield
\begin{align*}
&\ \ \ \frac{d}{dt}\|\widetilde{v}\|_{L^{3+\delta}}^{3+\delta}+\frac{1}{Re_1}\int_\Omega|\nabla_H\widetilde{v}|^2|\widetilde{v}|^{1+\delta}d\Omega+\frac{1}{Re_1}\int_\Omega\big|\nabla_H|\widetilde{v}|\big|^2|\widetilde{v}|^{1+\delta}d\Omega\\
&\ \ \ \ \ \ \ \ \ \qquad \ \ +\frac{1}{Re_2}\int_\Omega|\partial_z\widetilde{v}|^2|\widetilde{v}|^{1+\delta}d\Omega+\frac{1}{Re_2}\int_\Omega\big|\partial_z|\widetilde{v}|\big|^2|\widetilde{v}|^{1+\delta}d\Omega\\
&\leq C_\delta\big(\|v\|_{L^2}^2\|\nabla_H v\|_{L^2}^2+\|(\nabla_H v,T)\|_{L^2}^2+\|T\|_{L^2}\|\nabla_H T\|_{L^2}  +1\big)\big(\|\widetilde{v}\|_{L^{3+\delta}}^{3+\delta}+1\big),\label{59}
\end{align*}
from which, by the Gr\"onwall inequality, the conclusion follows.
\end{proof}

\subsection{Estimates of $(\partial_zv,\partial_zT)$}

Thanks to  Proposition \ref{pro3.2}, one can get the following estimate of $\partial_zv$.
\begin{proposition}\label{pro3.3}
For any positive time $\mathfrak{T}\in(0,t^*_\varepsilon)$, the following holds
\begin{eqnarray*}
\sup_{0\leq t\leq \mathfrak{T}}\|\partial_z v\|_{L^2}^2(t)+\int_{0}^{\mathfrak{T}}\left(\frac{1}{Re_1}\|\nabla_H\partial_z v\|_{L^2}^2+\frac{1}{Re_2}\|\partial_z^2 v\|_{L^2}^2\right)dt\\
\leq e^{C\int_{0}^{\mathfrak{T}}K_2(t)dt}\left(\|\partial_zv_0\|_{L^2}^2+C\int_{0}^{\mathfrak{T}}\|\nabla_H T\|_{L^2}^2dt\right) \label{37}
\end{eqnarray*}
for a positive constant $C$ depending only on $h, Re_1,$ and $Re_2$, where
\begin{equation*}
K_2(t):=\|v\|_{L^2}^2\|\nabla_Hv\|_{L^2}^2+\|\widetilde{v}\|_{L^{3+\delta}}^{2+\frac{6}{\delta}}+1.
\end{equation*}
\end{proposition}

\begin{proof}
Multiplying equation \eqref{19} with $-\partial_z^2 v$  and integrating over $\Omega$, it follows from integrating by parts that
\begin{equation*}
\begin{aligned}
&\ \ \  \frac{1}{2}\frac{d}{dt}\|\partial_z v\|_{L^2}^2+\frac{1}{Re_1}\|\nabla_H\partial_z v\|_{L^2}^2+\frac{1}{Re_2}\|\partial_z^2 v\|_{L^2}^2\\
&=\int_\Omega (v\cdot\nabla_H)v\cdot\partial_z^2v~d\Omega-\int_\Omega\left(\int_{-h}^z\nabla_H\cdot vd\xi\right)\partial_zv\cdot\partial_z^2v~d\Omega\\
&\qquad-\int_\Omega\left(\nabla_H\int_{-h}^zTd\xi\right)\cdot\partial_z^2v~d\Omega:=\sum_{i=1}^3I_i.
\end{aligned}
\end{equation*}
By integrating by parts, one deduces
\begin{equation*}
\begin{aligned}
I_1+I_2&=\int_\Omega\left[\nabla_H\cdot v|\partial_zv|^2-(\partial_zv\cdot\nabla_H)v\cdot\partial_zv\right]d\Omega \leq C\int_\Omega |\nabla_H\partial_z v| |v| |\partial_z v|d\Omega\\
&\leq C\int_\Omega |\nabla_H\partial_z v| |\widetilde{v}| |\partial_z v|d\Omega+C\int_\Omega |\nabla_H\partial_z v| |\overline{v}| |\partial_z v|d\Omega :=I_{11}+I_{12}.
\end{aligned}
\end{equation*}
$I_{12}$ can be estimated by using Lemma \ref{lemma1} and the Poincar\'e inequality as
\begin{equation*}
\begin{aligned}
|I_{12}|&\leq C\int_M\left(\int_{-h}^0|v|dz\right)\left(\int_{-h}^0|\nabla_H\partial_z v||\partial_z v|dz\right)dxdy\\
&\leq C\|\nabla_H\partial_zv\|_{L^2}\|\partial_z v\|_{L^2}^{\frac{1}{2}}\|\nabla_H\partial_zv\|_{L^2}^{\frac{1}{2}}\|v\|_{L^2}^{\frac{1}{2}}\|\nabla_Hv\|_{L^2}^{\frac{1}{2}}\\
&\leq \frac{1}{8Re_1}\|\nabla_H\partial_zv\|_{L^2}^2+C\|\partial_z v\|_{L^2}^2\|v\|_{L^2}^2\|\nabla_H v\|_{L^2}^2.
\end{aligned}
\end{equation*}
For $I_{11}$, it follows form the H\"older, Gagliardo--Nirenberg, Poincar\'e and Young inequalities that
\begin{equation*}
\begin{aligned}
|I_{11}|&\leq C\|\widetilde{v}\|_{L^{3+\delta}}\|\nabla_H\partial_z v\|_{L^2}\|\partial_z v\|_{L^{\frac{2(3+\delta)}{1+\delta}}}\\
&\leq C\|\widetilde{v}\|_{L^{3+\delta}}\|\nabla_H\partial_z v\|_{L^2}\|\nabla\partial_z v\|_{L^2}^{\frac{3}{3+\delta}}\|\partial_z v\|_{L^2}^{\frac{\delta}{3+\delta}}\\
&\leq \frac{1}{16Re_1}\|\nabla_H\partial_z v\|_{L^2}^2+C\|\widetilde{v}\|_{L^{3+\delta}}^2\|\nabla\partial_z v\|_{L^2}^{\frac{6}{3+\delta}}\|\partial_z v\|_{L^2}^{\frac{2\delta}{3+\delta}}\\
&\leq \frac{1}{8Re_1}\|\nabla_H\partial_z v\|_{L^2}^2+\frac{1}{8Re_2}\|\partial_z^2 v\|_{L^2}^2+C\|\widetilde{v}\|_{L^{3+\delta}}^{2+\frac{6}{\delta}}\|\partial_z v\|_{L^2}^2.
\end{aligned}
\end{equation*}
Integrating by parts and by the H\"older inequality, it follows 
\begin{equation*}
|I_3|\leq C\|\nabla_H T\|_{L^2}\|\partial_z v\|_{L^2}.
\end{equation*}
Thus, one has
\begin{equation*}
\begin{aligned}
&\ \ \  \frac{d}{dt}\|\partial_z v\|_{L^2}^2+\frac{1}{Re_1}\|\nabla_H\partial_z v\|_{L^2}^2+\frac{1}{Re_2}\|\partial_z^2 v\|_{L^2}^2\\
&\leq C\|\partial_z v\|_{L^2}^2\big(\|v\|_{L^2}^2\|\nabla_Hv\|_{L^2}^2+\|\widetilde{v}\|_{L^{3+\delta}}^{2+\frac{6}{\delta}}+1\big)+C\|\nabla_H T\|_{L^2}^2,\label{60}
\end{aligned}
\end{equation*}
from which, by the Gr\"onwall inequality, the conclusion follows.
\end{proof}

Next, one can see that $\partial_z T$ is bounded in $L^\infty(0,\mathfrak{T};L^2(\Omega))$.
\begin{proposition}\label{pro3.4}
For any positive time $\mathfrak{T}\in(0,t^*_\varepsilon)$, the following holds
\begin{align*}
&\sup_{0\leq t\leq \mathfrak{T}}\|\partial_z T\|_{L^2}^2(t)+ \frac{1}{R_T}\int_{0}^{\mathfrak{T}} \|\nabla_H\partial_z T\|_{L^2}^2 dt\\
\leq& e^{C\int_{0}^{\mathfrak{T}}K_3(t)dt}\left(\|\partial_zT_0\|_{L^2}^2+C\int_{0}^{\mathfrak{T}}\|\nabla_HT\|_{L^2}^2dt\right),\label{38}
\end{align*}
for a positive constant $C$ depending only on $R_T$, where
\begin{equation*}
\begin{aligned}
&K_3(t):=1+\|\nabla_Hv\|_{L^2}^2+\|\partial_z\nabla_Hv\|_{L^2}^2+\|\partial_zv\|_{L^2}^2\|\partial_z\nabla_Hv\|_{L^2}^2.
\end{aligned}
\end{equation*}
\end{proposition}

\begin{proof}
Multiplying \eqref{21} with $-\partial_z^2 T$, it follows from integration by parts that
\begin{eqnarray*}
\frac{1}{2}\frac{d}{dt}\|\partial_z T\|_{L^2}^2+\frac{1}{R_T}\|\nabla_H\partial_z T\|_{L^2}^2+\frac{\alpha_T}{R_T}\|\partial_zT\|_{L^2(\Gamma_s)}^2+\varepsilon\|\partial_z^2T\|_{L^2}^2\\
=-\int_\Omega\partial_zv\cdot\nabla_H T\partial_z T~d\Omega+\int_\Omega\nabla_H\cdot v|\partial_z T|^2~d\Omega
:=I_1+I_2.
\end{eqnarray*}
Using $|f|\leq \int_{-h}^0|\partial_zf|dz+\frac{1}{h}\int_{-h}^0|f|dz$, applying Lemma \ref{lemma1}, and by the Young and Poincar\'e inequalities, one deduces
\begin{equation*}
\begin{aligned}
|I_1|+|I_2|&\leq C\int_\Omega|\partial_z v|\left[\int_{-h}^0\Big(|\nabla_HT|+|\nabla_H\partial_zT|\Big)dz\right]|\partial_zT|d\Omega\\
&\ \ \ \  +C\int_\Omega\left[\int_{-h}^0\Big(|\nabla_Hv|+|\nabla_H\partial_zv|\Big)dz\right]|\partial_zT|^2d\Omega\\
&\leq C\|(\nabla_HT,\nabla_H\partial_zT)\|_{L^2} \|\partial_zT\|_{L^2}^{\frac{1}{2}}\|(\partial_zT,\nabla_H\partial_zT)\|_{L^2}^{\frac{1}{2}}
\|\partial_zv\|_{L^2}^{\frac{1}{2}}\|\nabla_H\partial_zv\|_{L^2}^{\frac{1}{2}}\\
&\ \ \ \  +C\big(\|\nabla_Hv\|_{L^2}+\|\nabla_H\partial_zv\|_{L^2}\big)\|\partial_zT\|_{L^2}\big(\|\partial_zT\|_{L^2}+\|\nabla_H\partial_zT\|_{L^2}\big)\\
&\leq C\big(1+\|\nabla_Hv\|_{L^2}^2+\|\partial_z\nabla_Hv\|_{L^2}^2+\|\partial_zv\|_{L^2}^2\|\partial_z\nabla_Hv\|_{L^2}^2\big)\|\partial_zT\|_{L^2}^2\\
&\ \ \ \ \ +\frac{1}{8R_T}\|\nabla_H\partial_zT\|_{L^2}^2+C\|\nabla_HT\|_{L^2}^2.
\end{aligned}
\end{equation*}
Therefore, 
\begin{equation*}
\begin{aligned}
&\ \ \  \frac{d}{dt}\|\partial_z T\|_{L^2}^2+\frac{1}{R_T}\|\nabla_H\partial_z T\|_{L^2}^2+\frac{\alpha_T}{R_T}\|\partial_zT\|_{L^2(\Gamma_s)}^2+\varepsilon\|\partial_z^2T\|_{L^2}^2\\
&\leq C\big(1+\|(\nabla_Hv,\partial_z\nabla_Hv)\|_{L^2}^2+\|\partial_zv\|_{L^2}^2\|\partial_z\nabla_Hv\|_{L^2}^2\big)\|\partial_zT\|_{L^2}^2
+C\|\nabla_HT\|_{L^2}^2,\label{61}
\end{aligned}
\end{equation*}
from which, by the Gr\"onwall inequality, the conclusion follows.
\end{proof}

\subsection{Estimates of $(\nabla_Hv,\nabla_HT)$ and $(\partial_tv,\partial_tT)$ }

We first consider the estimates of $\nabla_Hv$ and $\partial_tv$.
\begin{proposition}\label{pro3.5}
For any positive time $\mathfrak{T}\in(0,t^*_\varepsilon)$, the following holds
\begin{align*}
&\ \ \  \sup_{0\leq t\leq\mathfrak{T}}\|\nabla v\|_{L^2}^2+\int_0^{\mathfrak{T}}\|\partial_t v\|_{L^2}^2dt\\
&\leq C\left[\|\nabla v_0\|_{L^2}^2+\int_0^{\mathfrak{T}}\Big(K_3(t)+\|\nabla_HT\|_{L^2}^2+\|v\|_{L^2}^2\Big)dt\right]\\
&\qquad\qquad\times e^{C\int_0^{\mathfrak{T}}\big(K_3(t)+\|\nabla_HT\|_{L^2}^2+\|v\|_{L^2}^2(1+\|\nabla_Hv\|_{L^2}^2)\big)dt} \label{42}
\end{align*}
for a positive constant $C$ depending only on $Re_1$ and $Re_2$, and
\begin{equation*}
\int_0^{\mathfrak{T}}\|\nabla^2 v\|_{L^2}^2dt\leq J_1(\mathfrak{T}),\label{43}
\end{equation*}
where $J_1(t)$, a continuous function in $[0,\infty)$, is determined by  $h, Re_1, Re_2, R_T$, $\|v_0\|_{H^1}$ and $\|T_0\|_{L^2}$.
\end{proposition}
\begin{proof}
Rewrite \eqref{19} as
\begin{equation}
\begin{aligned}
\partial_t v-\frac{1}{Re_1}\Delta_H v-\frac{1}{Re_2}\partial_z^2 v+\nabla_H p_s=R,\\
v=0~~\textrm{on}~\Gamma_s,~~\partial_zv=0~~\textrm{on}~\Gamma_b~\textrm{and}~\Gamma_u,\label{39}
\end{aligned}
\end{equation}
where
\begin{equation*}
R:=-(v\cdot\nabla_H) v+\left(\int_{-h}^z\nabla_H\cdot vd\xi\right)\partial_z v+\nabla_H\int_{-h}^zTd\xi-fk\times v.
\end{equation*}
It follows by the elliptic estimate of hydrostatic Stokes equations (see \cite{Zia}) that
\begin{equation}
\begin{aligned}
\|\nabla^2 v\|_{L^2}\leq C(\|R\|_{L^2}+\|\partial_t v\|_{L^2}).\label{40}
\end{aligned}
\end{equation}
To estimate $R$,  one can see that
\begin{align*}
&\left\|\big(v\cdot\nabla_H\big)v-\int_{-h}^z\nabla_H\cdot vd\xi\partial_zv\right\|_{L^2}\\
\leq& C\left\|\int_{-h}^0|v|dz\nabla_Hv\right\|_{L^2}+\left\|\int_{-h}^0|\partial_zv|dz\nabla_Hv\right\|_{L^2}
+\left\|\int_{-h}^0|\nabla_Hv|dz\partial_zv\right\|_{L^2}.
\end{align*}
By the H\"older, Minkowski, and Ladyzhenskaya inequalities, one has
\begin{align*}
&\left\|\int_{-h}^0|\partial_zv|dz\nabla_Hv\right\|_{L^2}\\
=&\left[\int_M\left(\int_{-h}^0|\partial_zv|dz\right)^2\int_{-h}^0|\nabla_Hv|^2dz~dxdy\right]^{\frac{1}{2}}\\
\leq &\left[\int_M\left(\int_{-h}^0|\partial_zv|dz\right)^4dxdy\right]^{\frac{1}{4}}\left[\int_M\left(\int_{-h}^0|\nabla_Hv|^2dz\right)^2dxdy\right]^{\frac{1}{4}}\\
\leq &\int_{-h}^0\|\partial_zv\|_{L^4(M)}dz\left(\int_{-h}^0\|\nabla_Hv\|_{L^4(M)}^2dz\right)^{\frac{1}{2}}\\
\leq&C\left(\int_{-h}^0\|\nabla_Hv\|_{L^2(M)}\|\nabla_Hv\|_{H^1(M)}dz\right)^{\frac{1}{2}}\left(\int_{-h}^0\|\partial_zv\|_{L^2(M)}^{\frac{1}{2}}\|\nabla_H\partial_zv\|_{L^2(M)}^{\frac{1}{2}}dz\right)\\
\leq& C\Big(\|\nabla_H v\|_{L^2}\|\nabla_H^2v\|_{L^2}\|\partial_zv\|_{L^2}\|\nabla_H\partial_zv\|_{L^2}\Big)^{\frac{1}{2}}.
\end{align*}
Similarly
\begin{equation*}
\left\|\int_{-h}^0|v|dz\nabla_Hv\right\|_{L^2}\leq C\Big(\|v\|_{L^2}\|\nabla_Hv\|_{L^2}^2\|\nabla_H^2v\|_{L^2}\Big)^{\frac{1}{2}},
\end{equation*}
and
\begin{equation*}
\left\|\int_{-h}^0|\nabla_Hv|dz\partial_zv\right\|_{L^2}\leq C\Big(\|\nabla_H v\|_{L^2}\|\nabla_H^2v\|_{L^2}\|\partial_zv\|_{L^2}\|\nabla_H\partial_zv\|_{L^2}\Big)^{\frac{1}{2}}.
\end{equation*}

Therefore,
\begin{equation*}
\begin{aligned}
\|R\|_{L^2}&\leq C\|\nabla_Hv\|_{L^2}\|\nabla_H^2v\|_{L^2}^{\frac{1}{2}}\|v\|_{L^2}^{\frac{1}{2}}\\
&\ \ \ \  +C\|\nabla_Hv\|_{L^2}^{\frac{1}{2}}\|\nabla_H^2v\|_{L^2}^{\frac{1}{2}}\|\partial_zv\|_{L^2}^{\frac{1}{2}}\|\nabla_H\partial_zv\|_{L^2}^{\frac{1}{2}}\\
&\ \ \ \  +C\|\nabla_HT\|_{L^2}+C\|v\|_{L^2}.
\end{aligned}
\end{equation*}
Then, by \eqref{40} and the Young inequality, one can deduce
\begin{equation}
\begin{aligned}
\|\nabla^2 v\|_{L^2}\leq C&\|\nabla_H v\|_{L^2}^2\|v\|_{L^2}+C\|\nabla_Hv\|_{L^2}\|\partial_zv\|_{L^2}\|\nabla_H\partial_zv\|_{L^2}\\
&+C\|\nabla_HT\|_{L^2}+C\|v\|_{L^2}+C\|\partial_tv\|_{L^2}.\label{41}
\end{aligned}
\end{equation}
Multiplying \eqref{39} with $\partial_tv$, integrating over $\Omega$, and recalling that $\int_{-h}^0\nabla_H\cdot vdz=0$,
it follows from integrating by parts and \eqref{41} that
\begin{equation}
\begin{aligned}
& \frac{d}{dt}\Big(\frac{1}{Re_1}\|\nabla_H v\|_{L^2}^2+\frac{1}{Re_2}\|\partial_zv\|_{L^2}^2\Big)+\|\partial_t v\|_{L^2}^2\leq C\|R\|_{L^2}^2\\
\leq &C\|\nabla_Hv\|_{L^2}^2\|\nabla_H^2v\|_{L^2}\|v\|_{L^2} +C\|\nabla_Hv\|_{L^2}\|\nabla_H^2v\|_{L^2}\|\partial_zv\|_{L^2}\|\nabla_H\partial_zv\|_{L^2} \\ &+C\|\nabla_HT\|_{L^2}^2+C\|v\|_{L^2}^2\\
\leq& C\big(\|\nabla_Hv\|_{L^2}^2\|v\|_{L^2}^2+\|\nabla_H\partial_zv\|_{L^2}^2\|\partial_zv\|_{L^2}^2+\|\nabla_HT\|_{L^2}^2\\
& +\|v\|_{L^2}^2+1\big)\big(\|\nabla_Hv\|_{L^2}^2+1\big)+\frac{1}{2}\|\partial_tv\|_{L^2}^2.\label{52}
\end{aligned}
\end{equation}
Therefore,
\begin{equation*}
\begin{aligned}
&\ \ \ \frac{d}{dt}\Big(\frac{2}{Re_1}\|\nabla_H v\|_{L^2}^2+\frac{2}{Re_2}\|\partial_zv\|_{L^2}^2\Big)+\|\partial_t v\|_{L^2}^2\\
&\leq C\big[K_3+\|\nabla_HT\|_{L^2}^2+\|v\|_{L^2}^2(1+\|\nabla_Hv\|_{L^2}^2)\big](\|\nabla_Hv\|_{L^2}^2+1),
\end{aligned}
\end{equation*}
from which by the Gr\"onwall inequality, the first conclusion follows. The second conclusion follows from the first one
by using \eqref{41} and Propositions \ref{pro3.1} and \ref{pro3.3}.
\end{proof}

Next we give the estimates of $\nabla_HT$ and $\partial_tT$ in the following proposition.
\begin{proposition}\label{pro3.6}
For any positive time $\mathfrak{T}\in(0,t^*_\varepsilon)$, the following holds
\begin{align*}
&\sup_{0\leq t\leq \mathfrak{T}}\left(\|\nabla_H T\|_{L^2}^2+\alpha_T\|T\|_{L^2(\Gamma_s)}^2\right) +\frac{1}{R_T}\int_0^{\mathfrak{T}}\|\nabla_H^2T\|_{L^2}^2 dt\\ \leq& Ce^{C\int_0^{\mathfrak{T}}K_4(t)dt}\left(\|\nabla_HT_0\|_{L^2}^2+\|T_0\|_{L^2}^2+\int_0^{\mathfrak{T}}K_5(t)dt
+J_1(\mathfrak{T})\right) \label{44}
\end{align*}
for a positive constant $C$ depending only on $h, Re_1, Re_2,$ and $R_T$, where
\begin{align*}
K_4(t):=&1+\|(v,\partial_zv)\|_{L^2}^2\|(\nabla_Hv,\nabla_H\partial_zv)\|_{L^2}^2,\\
K_5(t):=&\|\nabla_Hv\|_{L^2}^2\left(\|\partial_zT\|_{L^2}^4+\|\partial_zT\|_{L^2}^2\|\nabla_H\partial_zT\|_{L^2}^2\right),
\end{align*}
and $J_1(t)$ is the function in Proposition \ref{pro3.5}. We also have
$$
\int_0^{\mathfrak{T}}\|\partial_tT\|_{L^2}^2dt\leq J_2(\mathfrak{T}),\label{45}
$$
where $J_2(t)$, a continuous function in $[0,\infty)$, is determined by $h, Re_1, Re_2,$ and $R_T$, and $\|(v_0,T_0)\|_{H^1}$.
\end{proposition}

\begin{proof}
Multiplying \eqref{21} with $-\Delta_HT$ and integrating over $\Omega$, it follows from integrating by parts, the Poincar\'e and Young inequalities, and Lemma \ref{lemma1} that
\begin{align*}
&\frac{1}{2}\frac{d}{dt}\Big(\|\nabla_HT\|_{L^2}^2+\alpha_T\|T\|_{L^2(\Gamma_s)}^2\Big)\\
&\qquad\qquad+\varepsilon\left(\|\partial_z\nabla_HT\|_{L^2}^2+\alpha_T\|\partial_zT\|_{L^2(\Gamma_s)}^2\right)+\frac{1}{R_T}\|\Delta_HT\|_{L^2}^2\\
=&\left(v\cdot\nabla_HT,\Delta_HT\right)_{L^2}-\left(\int_{-h}^z\nabla_H\cdot vd\xi\partial_zT,\Delta_HT\right)_{L^2}\\
\leq&\int_M\left(\frac{1}{h}\int_{-h}^0|v|dz+\int_{-h}^0|\partial_zv|dz\right)\left(\int_{-h}^0|\nabla_HT||\Delta_HT|dz\right)dxdy\\
&\qquad+\int_M\left(\int_{-h}^0|\nabla_Hv|dz\right)\left(\int_{-h}^0|\partial_zT||\Delta_HT|dz\right)dxdy\\
\leq&C\|\Delta_HT\|_{L^2}\|\nabla_HT\|_{L^2}^{\frac{1}{2}}\|(\nabla_HT,\nabla_H^2T)\|_{L^2}^{\frac{1}{2}} \|(v,\partial_zv)\|_{L^2}^{\frac{1}{2}}\|(\nabla_Hv,\nabla_H\partial_zv)\|_{L^2}^{\frac{1}{2}}\\
&+C\|\Delta_HT\|_{L^2}\|\partial_zT\|_{L^2}^{\frac{1}{2}}\|(\partial_zT,\nabla_H\partial_zT)\|_{L^2}^{\frac{1}{2}}
\|\nabla_Hv\|_{L^2}^{\frac{1}{2}}\|(\nabla_Hv,\nabla_H^2v)\|_{L^2}^{\frac{1}{2}}\\
\leq&C\big(\|(v,\partial_zv)\|_{L^2}^2\|(\nabla_Hv,\nabla_H\partial_zv)\|_{L^2}^2+1\big)\|\nabla_HT\|_{L^2}^2\\
& +C\|\nabla_Hv\|_{L^2}^2\big(\|\partial_zT\|_{L^2}^4+\|\partial_zT\|_{L^2}^2+\|\nabla_H\partial_zT\|_{L^2}^2\big)\\
& +C\|\Delta_Hv\|_{L^2}^2+\frac{1}{2R_T}\|\Delta_HT\|_{L^2}^2.
\end{align*}
Thus, one has
\begin{equation}
\begin{aligned}
\frac{d}{dt}\Big(\|\nabla_HT\|_{L^2}^2+\alpha_T\|T\|_{L^2(\Gamma_s)}^2\Big) +\frac{1}{R_T}\|\Delta_HT\|_{L^2}^2&\\
\leq CK_4(t)\|\nabla_HT\|_{L^2}^2+CK_5(t)+C\|\Delta_Hv\|_{L^2}^2,&\label{53}
\end{aligned}
\end{equation}
which implies the fist conclusion by using the Gr\"onwall inequality, Proposition \ref{pro3.5}, and the elliptic estimates.

It remains to estimate $\partial_tT$. Multiplying equation \eqref{21} with $\partial_tT$ and integrating over $\Omega$, it follows from integrating by parts that
\begin{equation*}
\begin{aligned}
\|\partial_tT\|_{L^2}^2+\frac{1}{2}\frac{d}{dt}\left(\frac{1}{R_T}\|\nabla_HT\|_{L^2}^2+\frac{\alpha_T}{R_T}
\|T\|_{L^2(\Gamma_s)}^2+\varepsilon\|\partial_zT\|_{L^2}^2\right)&\\
=-(v\cdot\nabla_HT,\partial_tT)_{L^2}+\left(\int_{-h}^z\nabla_H\cdot vd\xi\partial_zT,\partial_tT\right)_{L^2}.&
\end{aligned}
\end{equation*}
By Lemma \ref{lemma1}, the Young and Poincar\'e inequalities, one can obtain
\begin{align}
&\|\partial_tT\|_{L^2}^2+\frac{1}{2}\frac{d}{dt}\left(\frac{1}{R_T}\|\nabla_HT\|_{L^2}^2+\frac{\alpha_T}{R_T}
\|T\|_{L^2(\Gamma_s)}^2+\varepsilon\|\partial_zT\|_{L^2}^2\right)\nonumber\\
\leq&\int_M\left(\frac{1}{h}\int_{-h}^0|v|dz+\int_{-h}^0|\partial_zv|dz\right)
\left(\int_{-h}^0|\nabla_HT||\partial_tT|dz\right)dxdy\nonumber\\
& +\int_M\left(\int_{-h}^0|\nabla_Hv|dz\right)\left(\int_{-h}^0|\partial_zT||\partial_tT|dz\right)dxdy\nonumber\\
\leq&C\|\partial_tT\|_{L^2}\|(v,\partial_zv)\|_{L^2}^{\frac{1}{2}}\|(\nabla_Hv,\nabla_H\partial_zv)\|_{L^2}^{\frac{1}{2}}
\|\nabla_HT\|_{L^2}^{\frac{1}{2}}\|(\nabla_HT,\nabla_H^2T)\|_{L^2}^{\frac{1}{2}}\big)\label{54}\\
&+C\|\partial_tT\|_{L^2}\|\nabla_Hv\|_{L^2}^{\frac{1}{2}}\|\nabla_H^2v\|_{L^2}^{\frac{1}{2}}
\|\partial_zT\|_{L^2}^{\frac{1}{2}}\|(\partial_zT,\nabla_H\partial_zT)\|_{L^2}^{\frac{1}{2}} \nonumber\\
\leq&\frac{1}{2}\|\partial_tT\|_{L^2}^2+C\left(\|(v,\partial_zv)\|_{L^2}^2\|
(\nabla_Hv,\nabla_H\partial_zv)\|_{L^2}^2+1\right)\|\nabla_HT\|_{L^2}^2\nonumber\\
&+C\|\partial_zT\|_{L^2}^2\|(\partial_zT,\nabla_H\partial_zT)\|_{L^2}^2\|\nabla_Hv\|_{L^2}^2
+C\big(\|\nabla_H^2T\|_{L^2}^2+\|\nabla_H^2v\|_{L^2}^2\big),\nonumber\\
\leq&\frac{1}{2}\|\partial_tT\|_{L^2}^2+CK_4(t)\|\nabla_HT\|_{L^2}^2+CK_5(t)
+C\big(\|\nabla_H^2T\|_{L^2}^2+\|\nabla_H^2v\|_{L^2}^2\big).\nonumber
\end{align}
Then, by the Gr\"onwall inequality, it follows 
\begin{equation*}
\begin{aligned}
\int_0^{\mathfrak{T}}\|\partial_tT\|_{L^2}^2dt+\sup_{0\leq t\leq \mathfrak{T}} \Big(\frac{1}{R_T}\|\nabla_HT\|_{L^2}^2+\frac{\alpha_T}{R_T}\|T\|_{L^2(\Gamma_s)}^2+\varepsilon\|\partial_zT\|_{L^2}^2\Big)&\\
\leq Ce^{C\int_0^{\mathfrak{T}}K_4(t)dt}\Big[\|T_0\|_{H^1}^2+\int_0^{\mathfrak{T}}\Big(K_5(t)+\|\Delta_HT\|_{L^2}^2\Big)dt
+J_1(\mathfrak{T})\Big]&,
\end{aligned}
\end{equation*}
from which, by the first conclusion, the second conclusion follows.
\end{proof}

\subsection{Existence of the global $H^1$ strong solution}

Base on the estimates in the previous three subsections, we have the following global existence theorem for $H^1$ initial data.
\begin{theorem}\label{thm3.1}
Let $(v_0,T_0)\in \mathcal V\times H^1(\Omega)$. Then, for any $\varepsilon>0$, there is a unique global strong solution $(v,T)$ to system \eqref{19}--\eqref{21}, subject to the boundary and initial conditions \eqref{22}--\eqref{24}, satisfying the following estimate
\begin{equation*}
\sup_{0\leq t\leq \mathfrak{T}}\|(v,T)\|_{H^1}^2+\int_0^{\mathfrak{T}}\Big(\|v\|_{H^2}^2
+\|(\Delta_HT,\nabla_H\partial_zT,\partial_tv,\partial_tT)\|_{L^2}^2\Big)dt\leq J_3(\mathfrak{T}),
\end{equation*}
where $J_3(t)$, a continuous function, is determined by $h, Re_1, Re_2, R_T,$ and $\|(v_0,T_0)\|_{H^1}$.
\end{theorem}

\begin{proof}
As before, let $(v,T)$ be the unique strong solution corresponding to the initial data $(v_0,T_0)$, with maximal interval of existence $[0,t^*_\varepsilon)$. If $t^*_\varepsilon<\infty$, then 
\begin{equation*}
\limsup_{t\rightarrow (t^*_\varepsilon)^-}\big(\|v\|_{H^1}(t)+\|T\|_{H^1}(t)\big)=\infty.
\end{equation*}
However, this contradicts to Proposition \ref{pro3.1}--\ref{pro3.6}.  Therefore, $t^*_\varepsilon=\infty$ and $(v,T)$ is a global solution. The corresponding estimate follows from Propositions \ref{pro3.1}--\ref{pro3.6}.
\end{proof}

\section{A priori estimates depending only on $\|(v_0, T_0, \partial_zv_0,\partial_zT_0)\|_{L^2}$}

Note that the estimate of $\partial_zv$ carried out
in the previous section depends not only on the assumption $\partial_zv_0\in L^2(\Omega)$ but also on $\widetilde{v}_0\in L^{3+\delta}(\Omega)$, for $0<\delta\leq 1$; however, these assumptions
are not satisfied by the initial data of $z$-weak solutions. As a result,
the estimates established in the previous section are not sufficient and some extra estimates are required to prove the existence of $z$-weak solutions.

In this section, we establish the $L^\infty(0,\mathfrak T; L^2)$ estimate of $(\partial_zv,\partial_zT)$ and some $t$-weighted $L^\infty(0,\mathfrak T; L^2)$ estimate of $(\nabla_H v,\nabla_H T)$ to the strong solution $(v,T)$ established in Theorem \ref{thm3.1}. These estimates depend
only on $\|(T_0,v_0,\partial_zv_0,\partial_zT_0)\|_{L^2}$, rather than on $\|(v_0, T_0)\|_{H^1}$, even though we
still assume that $(v_0, T_0)\in H^1$.

First, one has the following local in time a priori estimate for $\partial_zv$.

\begin{proposition}\label{pro4.1}
Let $(v,T)$ be a global strong solution to system \eqref{19}--\eqref{21}, subject to \eqref{22}--\eqref{24}.
There is a small positive number $\delta_0\leq1$ depending only on $h, Re_1, Re_2, R_T$, and $\|(v_0, T_0)\|_{L^2}$, such
that if
\begin{equation*}
\sup_{X\in M}\int_{-h}^0\int_{D_{2r_0}(X)\cap M}|\partial_zv_0|^2dxdydz\leq \delta_0^2\label{46}
\end{equation*}
for some positive number $r_0\leq 1$, where $D_{2r_0}(X)$ is the disk of radius $2r_0$ and centered at $X\in M$,
then the following estimate holds
\begin{equation*}
\sup_{0\leq t\leq t_0^*}\|\partial_zv\|_{L^2}^2(t)+\int_0^{t_0^*}\Big(\|\nabla_H\partial_zv\|_{L^2}^2+\|\partial_z^2v\|_{L^2}^2\Big)dt\leq C_K\label{50}
\end{equation*}
for some positive constant $C_K$ depending only on $h, Re_1, Re_2,$ $R_T$, $\delta_0$, $r_0, M$, and $\|(v_0, T_0)\|_{L^2}$, where $t_0^*=\min\left\{1,\frac{r_0^4\delta_0^2}{C_0}
\right\}$ with $C_0$ depending only on $h, Re_1, Re_2, R_T,$ and $\|(v_0,T_0)\|_{L^2}$.
\end{proposition}

\begin{proof}
  The proof of this proposition is exactly the same as that of Proposition 3.2 in \cite{Cao1} and thus is omitted
  here.
\end{proof}

Next, the short time estimate of $\partial_zT$ is given in the following proposition.

\begin{proposition}\label{pro4.2}
Let $(v,T)$ be as in Theorem \ref{thm3.1} and let $t_0^*$ and $C_K$ be the positive constants in Proposition \ref{pro4.1}.
Then, one has
\begin{align*}
\sup_{0\leq t\leq t_0^*}\|\partial_z T\|_{L^2}^2(t) +\frac{1}{R_T}\int_{0}^{t_0^*} \|\nabla_H\partial_z T\|_{L^2}^2 dt
\leq e^{C(1+C_K^2)}\big(\|\partial_zT_0\|_{L^2}^2+1\big) \label{51}
\end{align*}
for some positive constant $C$ depends only on $h, Re_1, Re_2, R_T, C_K$, and $\|(v_0,T_0)\|_{L^2}$.
\end{proposition}

\begin{proof}
It follows from Proposition \ref{pro3.1} and Proposition \ref{pro4.1} that
\begin{equation*}
\int_0^{t_0^*}\|\nabla_HT\|_{L^2}^2dt\leq C,\qquad \int_0^{t_0^*}K_3(t)dt\leq C+C_K^2,
\end{equation*}
where $C$ depends only on $h, Re_1, Re_2, R_T$ and $\|(v_0,T_0)\|_{L^2}$. With the aid of this and applying Proposition \ref{pro3.4} for
$\mathfrak{T}=t^*_0$, one gets the conclusion.
\end{proof}

Then, we have the $t$-weight estimates of $\nabla_Hv$ and $\partial_tv$.

\begin{proposition}\label{pro4.3}
Let $(v,T)$ be as in Theorem \ref{thm3.1} and let $t_0^*$ and $C_K$ be the positive constants in Proposition \ref{pro4.1}. Then, the following hold
\begin{equation*}
\sup_{0\leq t\leq t_0^*}\|\sqrt t\nabla_H v\|_{L^2}^2+\int_0^{t_0^*}t(\|\partial_t v\|_{L^2}^2+\|\nabla^2v\|_{L^2}^2)dt\leq C \label{55}
\end{equation*}
for a positive constant $C$ depending only on $h, Re_1, Re_2, R_T, C_K$, and $\|(v_0,T_0)\|_{L^2}$.
\end{proposition}

\begin{proof}
Multiplying \eqref{52} with $t$ yields
\begin{eqnarray*}
&&\frac{d}{dt}\left(\frac{2t}{Re_1}\|\nabla_H v\|_{L^2}^2+\frac{2t}{Re_2}\|\partial_zv\|_{L^2}^2\right)+t\|\partial_t v\|_{L^2}^2\\
&\leq& C\big(\|\nabla_Hv\|_{L^2}^2\|v\|_{L^2}^2+\|\nabla_H\partial_zv\|_{L^2}^2\|\partial_zv\|_{L^2}^2+\|\nabla_HT\|_{L^2}^2\\
&&+\|v\|_{L^2}^2+1\big)\big(t\|\nabla_Hv\|_{L^2}^2+t\big)+C\|v\|_{H^1}^2,
\end{eqnarray*}
from which, by the Gr\"onwall inequality and applying Proposition \ref{pro3.1} and Proposition \ref{pro4.1}, one gets
\begin{equation*}
\sup_{0\leq t\leq t_0^*}\big(t\|\nabla_H v\|_{L^2}^2+t\|\partial_zv\|_{L^2}^2\big)+\int_0^{t_0^*}t \|\partial_t v\|_{L^2}^2dt\leq C. \label{55}
\end{equation*}
Thanks to this, using \eqref{41}, and applying Proposition \ref{pro3.1} and Proposition \ref{pro4.1} again, one gets the
estimate for $\nabla^2v$ and, thus, proves the conclusion.
\end{proof}

And finally, one can deduce the $t$-weight estimates of $\nabla_HT$ and $\partial_tT$ in the below.

\begin{proposition}\label{pro4.4}
Let $(v,T)$ be as in Theorem \ref{thm3.1} and let $t_0^*$ and $C_K$ be the positive constants in Proposition \ref{pro4.1}. Then, the following holds
\begin{align*}
&\sup_{0\leq t\leq t_0^*}\|\sqrt t\nabla_H T\|_{L^2}^2
+\int_0^{t_0^*}t(\|\nabla_H^2T\|_{L^2}^2+\|\partial_tT\|_{L^2}^2)dt\leq  C\label{57}
\end{align*}
for a positive constant $C$ depending only on $h, Re_1, Re_2,$ $R_T,$ $\|(v_0,T_0,\partial_zT_0)\|_{L^2}$, and $C_K$.
\end{proposition}

\begin{proof}
Multiplying \eqref{53} with $t$, one has
\begin{equation}
\label{ADDLI0}
\begin{aligned}
&\frac{d}{dt}\Big(t\|\nabla_HT\|_{L^2}^2+t\alpha_T\|T\|_{L^2(\Gamma_s)}^2\Big)
+\frac{t}{R_T}\|\Delta_HT\|_{L^2}^2\\
\leq&C\Big(tK_4(t)\|\nabla_HT\|_{L^2}^2+tK_5(t)+t\|\Delta_Hv\|_{L^2}^2+\|\nabla_HT\|_{L^2}^2+\|T\|_{L^2}^2\Big).
\end{aligned}
\end{equation}
By Proposition \ref{pro3.1} and Propositions \ref{pro4.1}--\ref{pro4.3}, one can check that
\begin{equation}
\int_0^{t^*_0}(K_4(t)+tK_5(t))dt\leq C.\label{ADDLI1}
\end{equation}
Thanks to this, applying the Gr\"onwall inequality to (\ref{ADDLI0}), and using Proposition \ref{pro3.1} and Proposition \ref{pro4.3}, it follows
\begin{equation}
\label{ADDLI2}
\begin{aligned}
&\sup_{0\leq t\leq t_0^*}\Big(t\|\nabla_H T\|_{L^2}^2+t\alpha_T\|T\|_{L^2(\Gamma_s)}^2\Big)
+\int_0^{t_0^*}t\|\Delta_HT\|_{L^2}^2dt\leq  C.
\end{aligned}
\end{equation}

It remains to estimate $\partial_tT$.
Multiplying \eqref{54} with $t$ yields 
\begin{equation*}
\begin{aligned}
&t\|\partial_tT\|_{L^2}^2+\frac{d}{dt}\left(\frac{t}{R_T}\|\nabla_HT\|_{L^2}^2+\frac{t\alpha_T}{R_T}\|T\|_{L^2(\Gamma_s)}^2+t\varepsilon\|\partial_zT\|_{L^2}^2\right)\\
\leq&CtK_4(t)\|\nabla_HT\|_{L^2}^2+CtK_5(t)+C\big(t\|\Delta_HT\|_{L^2}^2+t\|\Delta_Hv\|_{L^2}^2\big)\\
&\ \ \ \ \ +C\big(\|\nabla_HT\|_{L^2}^2+\|T\|_{L^2}^2+\varepsilon\|\partial_zT\|_{L^2}^2\big),
\end{aligned}
\end{equation*}
from which, by the Gr\"onwall inequality, (\ref{ADDLI1}), (\ref{ADDLI2}), Proposition \ref{pro3.1}, and Proposition \ref{pro4.3}, it follows
\begin{equation*}
\begin{aligned}
\int_0^{t_0^*}t\|\partial_tT\|_{L^2}^2dt+\sup_{0\leq t\leq t_0^*} \left(\frac{t}{R_T}\|\nabla_HT\|_{L^2}^2+\frac{t\alpha_T}{R_T}\|T\|_{L^2(\Gamma_s)}^2+t\varepsilon\|\partial_zT\|_{L^2}^2\right)
\leq C.
\end{aligned}
\end{equation*}
This completes the proof.
\end{proof}

Combining the local-in-time
estimates in the previous propositions with the
global in time estimates in Theorem \ref{thm3.1}, one can get the global in time estimates depending only on
$\|(v_0, T_0, \partial_zv_0, \partial_zT_0)\|_{L^2}$. 

In fact, we have the following corollary.

\begin{corollary}\label{cor5.1}
Let $(v,T)$ be as in Theorem \ref{thm3.1} and let $\delta_0$ and $r_0$ be the positive constants in Proposition \ref{pro4.1}. Assume that (\ref{50}) holds. Then, it holds that
\begin{equation*}
\begin{aligned}
\int_{0}^{\mathfrak{T}}\Big(\|(\nabla v,\nabla\partial_z v,&\nabla_HT,\nabla_H\partial_zT)\|_{L^2}^2+t\|(\nabla_H^2v,\nabla_H^2T,\partial_tv,\partial_tT)\|_{L^2}^2\Big)dt\\
+\sup_{0\leq t\leq \mathfrak{T}}&\Big(\|(v,T,\partial_zv,\partial_zT)\|_{L^2}^2+t\|(\nabla_Hv,\nabla_HT)\|_{L^2}^2\Big)\leq J_4(\mathfrak{T})\label{65}
\end{aligned}
\end{equation*}
for any positive time $\mathfrak T$,
where $J_4$ is a continuous function determined only by $h, Re_1$, $Re_2$, $R_T$, $\delta_0, r_0, M,$ and $\|(v_0,T_0,\partial_zT_0)\|_{L^2}^2$.
\end{corollary}

\begin{proof}
Let $t_0^*$ be the positive time in Proposition \ref{pro4.1} and set $\sigma_0=\frac{t_0^*}{2}$.
By Proposition \ref{pro3.1} and Propositions \ref{pro4.1}--\ref{pro4.4}, one has
$\left\|(v,T)\right\|_{H^1}(\sigma_0) \leq C$
for a positive constant $C$ depending only on $h$, $Re_1$, $Re_2$, $R_T$,
$\delta_0$, $r_0$, $M,$ and $\|(v_0,T_0,\partial_zT_0)\|_{L^2}^2$.
Thanks to this, by viewing $\sigma_0$ as the initial time, one can apply Theorem \ref{thm3.1} to get
$$
\sup_{\sigma_0\leq t\leq \mathfrak{T}}\|(v,T)\|_{H^1}^2+\int_{\sigma_0}^{\mathfrak{T}}\Big(\|(\nabla^2 v,\nabla_H\nabla T,\partial_tv,\partial_tT)\|_{L^2}^2+\varepsilon\|\partial_z^2T\|_{L^2}^2\Big)dt\leq J_3^*(\mathfrak{T}) \label{64}
$$
for any $\mathfrak T\geq\sigma_0$, where $J_3^*$ is a continuous function determined only by $h$, $Re_1$, $Re_2$, $R_T$, $\delta_0$, $r_0$, $M$, and
$\|(v_0,T_0,\partial_zT_0)\|_{L^2}^2$. Thanks to the above and applying Propositions \ref{pro3.1} and \ref{pro4.2}--\ref{pro4.4} again, the conclusion follows.
\end{proof}

We also have the $L^{\frac{4}{3}}(0,\mathfrak{T};\mathcal V^*)\times L^{\frac{4}{3}}(0,\mathfrak{T}; (H^1(\Omega))^*)$ estimate for $(\partial_tv,\partial_tT)$ stated in the following proposition, where $\mathcal V^*$ denotes the dual space of $\mathcal V$.

\begin{proposition}\label{pro5.1}
Let $(v,T)$ be as in Theorem \ref{thm3.1} and let $\delta_0$ and $r_0$ be the positive constants in Proposition \ref{pro4.1}. Then, the following estimates hold
$$
\|\partial_tv\|_{L^{\frac{4}{3}}(0,\mathfrak{T};\mathcal V^*)}\leq J_5(\mathfrak{T}),\quad \|\partial_tT\|_{L^{\frac{4}{3}}(0,\mathfrak{T};(H^{1}(\Omega))^*)}\leq J_6(\mathfrak{T}),\label{67}
$$
for any positive time $\mathfrak T$, where both $J_5(t)$ and $J_6(t)$ are continuous functions determined only by $h, Re_1, Re_2, R_T, \delta_0, r_0, M$, and $\|(v_0,T_0,\partial_zT_0)\|_{L^2}^2$.
\end{proposition}

\begin{proof}
Taking arbitrary $\psi \in L^4(0,\mathfrak{T};\mathcal V)$ as a testing function in \eqref{19}, it follows from integrating by parts that
\begin{equation*}
\begin{aligned}
\int_0^{\mathfrak{T}}(\partial_tv,\psi)_{L^2}dt=&-\int_0^{\mathfrak{T}}\frac{1}{Re_2}(\nabla_Hv,\nabla_H\psi)_{L^2}dt-
\frac{1}{Re_2}\int_0^{\mathfrak{T}}(\partial_zv,\partial_z\psi)_{L^2}dt\\
&-\int_0^{\mathfrak{T}}\big((v\cdot\nabla_H)v,\psi\big)_{L^2}dt+\int_0^{\mathfrak{T}}\left(\int_{-h}^z\nabla_H\cdot vd\xi\partial_z v,\psi\right)_{L^2}dt\\
&-\int_0^{\mathfrak{T}}(fk\times v,\psi)_{L^2}dt
+\int_0^{\mathfrak{T}}\left(\nabla_H\int_{-h}^zTd\xi,\psi\right)_{L^2}dt:=\sum_{i=1}^6I_i.
\end{aligned}
\end{equation*}
By the H\"older inequality, one has
\begin{equation*}
|I_1+I_2|\leq C\|\nabla v\|_{L^2(L^2)}\|\nabla\psi\|_{L^2(L^2)},
\end{equation*}
where we use $L^p(X)$ to denote $L^p(0,\mathfrak{T};X)$ for some Banach space $X$. It follows from Lemma \ref{lemma1} that
\begin{equation*}
\begin{aligned}
&|I_3+I_4|\\
\leq& C\int_0^{\mathfrak{T}}\int_\Omega \Bigg[\int_{-h}^0|\nabla_Hv|dz|\partial_zv||\psi| +\int_{-h}^0\Big(|\partial_zv|+|v|\Big)dz|\nabla_Hv||\psi|\Bigg]d\Omega dt\\
\leq& C\int_0^{\mathfrak{T}}\Big(\|\partial_zv\|_{L^2}^{\frac{1}{2}}\|(\partial_zv,\nabla_H\partial_zv)\|_{L^2}^{\frac{1}{2}}
 +\|v\|_{L^2}^{\frac{1}{2}}\|(v,\nabla_H v)\|_{L^2}^{\frac{1}{2}}\Big)\\
& \times\|\nabla_Hv \|_{L^2}\|\psi\|_{L^2}^{\frac{1}{2}} \|(\psi,\nabla_H\psi)\|_{L^2}^{\frac{1}{2}} dt\\
\leq& C\Big(\|v\|_{L^\infty(L^2)}^{\frac{1}{2}} \|(v,\nabla_Hv)\|_{L^2(L^2)}^{\frac{1}{2}}
+\|\partial_zv\|_{L^\infty(L^2)}^{\frac{1}{2}}\|(\partial_zv,\partial_z\nabla_Hv)\|_{L^2(L^2)}^{\frac{1}{2}}\Big)\\
&  \times\|\nabla_H v\|_{L^2(L^2)} \|(\psi,\nabla_H\psi)\|_{L^4(L^2)}.
\end{aligned}
\end{equation*}
Applying the H\"older inequality yields
\begin{equation*}
\begin{aligned}
|I_5+I_6|\leq C\left(\|v\|_{L^2(L^2)}+\|\nabla_HT\|_{L^2(L^2)}\right)\|\psi\|_{L^2(L^2)}.
\end{aligned}
\end{equation*}
Thus, by Corollary \ref{cor5.1}, one gets
\begin{equation*}
\left|\int_0^{\mathfrak{T}}(\partial_tv,\psi)_{L^2}dt\right|\leq J_5(\mathfrak{T})\|\psi\|_{L^4(\mathcal V)},\quad
\forall\psi\in L^4(0,\mathfrak{T};\mathcal V)
\end{equation*}
and, as a result, $\|\partial_tv\|_{L^{\frac{4}{3}}(0,\mathfrak{T};\mathcal V^*)}\leq J_5(\mathfrak{T}).$

Similarly, taking arbitrary $\phi\in L^4(0,\mathfrak{T};H^1(\Omega))$ as a testing function in \eqref{21}, it follows from integrating by parts that
\begin{equation*}
\begin{aligned}
\int_0^{\mathfrak{T}}(\partial_tT,\phi)_{L^2}dt=&\int_0^{\mathfrak{T}}\left[-\frac{1}{R_T}(\nabla_HT,\nabla_H \phi)_{L^2}-\frac{\alpha_T}{R_T}(T,\phi)_{L^2(\Gamma_s)}-\varepsilon(\partial_zT,\partial_z\phi)_{L^2}\right]dt\\
&+\int_0^{\mathfrak{T}}\left[-(v\cdot\nabla_HT,\phi)_{L^2} +\left(\int_{-h}^z\nabla_H\cdot vd\xi\partial_z T,\phi\right)_{L^2}\right]dt.
\end{aligned}
\end{equation*}
By the H\"older inequality, one has
\begin{equation*}
\begin{aligned}
&\left|\int_0^{\mathfrak{T}}\left[\frac{1}{R_T}(\nabla_HT,\nabla_H\phi)_{L^2}+\varepsilon(\partial_zT,\partial_z\phi)_{L^2}\right]dt\right|\\
\leq&\frac{1}{R_T}\|\nabla_HT\|_{L^2(L^2)}\|\nabla_H\phi\|_{L^2(L^2)}+\varepsilon\|\partial_zT\|_{L^2(L^2)}\|\partial_z\phi\|_{L^2(L^2)}.
\end{aligned}
\end{equation*}
Applying the trace inequality leads to
\begin{equation*}
\begin{aligned}
&\left|\int_0^{\mathfrak{T}}\frac{\alpha_T}{R_T}(T,\phi)_{L^2(\Gamma_s)}dt\right|\leq C\|T\|_{L^2(L^2(\Gamma_s))}\|\phi\|_{L^2(L^2(\Gamma_s))}\\
\leq &C\left(\|T\|_{L^2(L^2)}+\|\nabla_HT\|_{L^2(L^2)}\right)\left(\|\phi\|_{L^2(L^2)}+\|\nabla_H\phi\|_{L^2(L^2)}\right).
\end{aligned}
\end{equation*}
By Lemma \ref{lemma1}, the nonlinear terms can be estimated as
\begin{equation*}
\begin{aligned}
&\left|\int_0^{\mathfrak{T}}\left[(v\cdot\nabla_HT,\phi)_{L^2}-\left(\int_{-h}^z\nabla_H\cdot vd\xi\partial_zT,\phi\right)_{L^2}\right]dt\right|\\
\leq&C\int_0^{\mathfrak{T}}\Bigg[\int_M\int_{-h}^0\Big(|v|+|\partial_zv|\Big)dz\int_{-h}^0|\nabla_HT||\phi|dzdxdy\\
& +\int_M\int_{-h}^0|\nabla_Hv|dz\int_{-h}^0|\partial_zT||\phi|dzdxdy\Bigg]dt\\
\leq&C\int_0^{\mathfrak{T}} \Big(\|\nabla_HT\|_{L^2}\|(v,\partial_zv)\|_{L^2}^{\frac{1}{2}}
\|(v,\partial_zv,\nabla_Hv,\nabla_H\partial_zv)\|_{L^2}^{\frac{1}{2}} \\
&+\|\nabla_Hv\|_{L^2}\|\partial_zT\|_{L^2}^{\frac{1}{2}}\|(\partial_zT,\nabla_H\partial_zT)\|_{L^2}^{\frac{1}{2}}\Big) \|\phi\|_{L^2}^{\frac{1}{2}}\|(\phi,\nabla_H\phi)\|_{L^2}^{\frac{1}{2}} dt\\
\leq&C\Big(\|\nabla_HT\|_{L^2(L^2)}\|(v,\partial_zv)\|_{L^\infty(L^2)}^{\frac{1}{2}}
\|(v,\partial_zv,\nabla_Hv,\nabla_H\partial_zv)\|_{L^2(L^2)}^{\frac{1}{2}}\\
&+\|\nabla_Hv\|_{L^2(L^2)}\|\partial_zT\|_{L^\infty(L^2)}^{\frac{1}{2}}
\|(\partial_zT,\nabla_H\partial_zT)\|_{L^2(L^2)}^{\frac{1}{2}}\Big)\|(\phi,\nabla_H\phi)\|_{L^4(L^2)}.
\end{aligned}
\end{equation*}
By Corollary \ref{cor5.1} and combining the above estimates, one has
\begin{equation*}
\left|\int_0^{\mathfrak{T}}(\partial_tT,\phi)_{L^2}dt\right|\leq J_6(\mathfrak{T})\|\phi\|_{L^4(H^1)},\quad \forall\phi\in L^4(0,\mathfrak{T};H^1(\Omega)),
\end{equation*}
and, thus, $\|\partial_tT\|_{L^{\frac{4}{3}}(0,\mathfrak{T};(H^1(\Omega))^*)}\leq J_6(\mathfrak{T}).$
\end{proof}

\section{Proof of Theorem \ref{thm1}}

We are now ready to give the proof of Theorem \ref{thm1}

\begin{proof}[\textbf{Proof of Theorem \ref{thm1}}]
\textbf{(i) Existence.}
Take a sequence $\{(v_{0n},T_{0n})\}_{n=1}^\infty \subseteq \mathcal V\times H^1(\Omega)$, such that
\begin{equation*}
(v_{0n},T_{0n},\partial_zv_{0n},\partial_zT_{0n})\rightarrow(v_0,T_0,\partial_zv_0,\partial_zT_0)\quad \textrm{in}~L^2(\Omega),  \quad\textrm{ as}~n\rightarrow \infty.
\end{equation*}
Let $\delta_0$ be the constant stated in Proposition \ref{pro4.1} and choose $r_0\in(0,1)$ sufficiently small, such that
\begin{equation*}
\sup_{X\in M}\int_{-h}^0\int_{D_{2r_0}(X)\cap M}|\partial_zv_0|^2dxdydz\leq \frac{\delta_0^2}{2}.
\end{equation*}
Since $\partial_zv_{0n}\rightarrow \partial_zv_0$ in $L^2(\Omega)$, as $n\rightarrow\infty$, there exists an integer
$N_0$, such that
\begin{equation*}
\sup_{X\in M}\int_{-h}^0\int_{D_{2r_0}(X)\cap M}|
\partial_zv_{0n}|^2dxdydz\leq \delta_0^2,\quad\forall n\geq N_0.
\end{equation*}

By Theorem \ref{thm3.1}, for each $n$, there is a unique global solution $(v_n,T_n)$ to system (\ref{19})--(\ref{21}), with $\varepsilon=\frac{1}{n^2}$, subject to (\ref{23})--(\ref{24}), and with initial data $(v_{0n},T_{0n})$. By Corollary \ref{cor5.1} and Proposition \ref{pro5.1}, the following estimates hold
\begin{eqnarray*}
\sup_{0\leq t\leq \mathfrak{T}} \|(v_{n},T_n,\partial_zv_n,\partial_zT_n)\|_{L^2}^2
+\int_{0}^{\mathfrak{T}} \|(\nabla v_n,\nabla\partial_zv_n,\nabla_HT_n,\nabla_H\partial_zT_n)\|_{L^2}^2 dt\leq C,\\
\sup_{0\leq t\leq \mathfrak{T}}\Big(t\|(\nabla_Hv_n,\nabla_HT_n)\|_{L^2}^2\Big)+
\int_{0}^{\mathfrak{T}} t\|(\nabla_H^2v_n,\nabla_H^2T_n,\partial_tv_n,\partial_t
T_n)\|_{L^2}^2 dt\leq C,\label{69}
\end{eqnarray*}
and
$$
\|\partial_tv_n\|_{L^\frac43(0,\mathfrak T; \mathcal V^*)}+\|\partial_tT\|_{L^\frac43(0,\mathfrak T; (H^1(\Omega))^*)}
\leq C,
$$
for any $n\geq N_0$ and for a positive constant $C$ independent of $n\geq N_0$.

Thanks to the above estimates, by the Aubin--Lions lemma, i.e., Lemma \ref{lemma2}, and by Cantor's diagonal arguments,
there exists a pair $(v,T)$ and a subsequence of $\{(v_n,T_n)\}_{n=1}^\infty$, still denoted by $\{(v_n,T_n)\}_{n=1}^\infty$, such that
\begin{equation*}
\begin{aligned}
&(v_n,\partial_zv_n,T_n,\partial_z T_n)\stackrel{\ast}{\rightharpoonup}(v,\partial_zv,T,\partial_zT)\quad \textrm{in}~L^\infty(0,\mathfrak{T};L^2(\Omega)),\\
&(\nabla v_n,\nabla\partial_zv_n)
\rightharpoonup(\nabla v, \nabla\partial_zv)\quad\textrm{in}~L^2(\Omega\times(0,\mathfrak{T})),\\
&(\nabla_HT_n,\nabla_H\partial_zT_n)
\rightharpoonup(\nabla_HT,\nabla_H\partial_zT)\quad\textrm{in}~L^2(\Omega\times(0,\mathfrak{T})),\\
&\sqrt t(\nabla v_n,\nabla T_n)\stackrel{\ast}{\rightharpoonup}\sqrt t(\nabla v,\nabla T)\quad\textrm{in}~L^\infty(0,\mathfrak{T};L^2(\Omega)),\\
&\sqrt t(\nabla_H^2v_n,\nabla_H^2T_n)\rightharpoonup \sqrt t(\nabla_H^2v,\nabla_H^2T)\quad \textrm{in}~L^2(\Omega\times(0,\mathfrak{T})),\\
&\sqrt t(\partial_tv_n,\partial_tT_n)\rightharpoonup \sqrt t(\partial_t v,\partial_tT)\quad \textrm{in}~L^2(\Omega\times(0,\mathfrak{T})),\\
& v_n\rightarrow v\quad \textrm{in}~C([\delta,\mathfrak{T}];\mathcal H)\cap C([0,\mathfrak T]; \mathcal V^*),\\
&T_n\rightarrow T\quad \textrm{in}~C([\delta,\mathfrak{T}];L^2(\Omega))\cap C([0,\mathfrak T]; (H^1(\Omega))^*),
\end{aligned}
\end{equation*}
as $n\rightarrow\infty$, for any $0<\delta<\mathfrak{T}<\infty$, where $\rightharpoonup$ and
$\stackrel{\ast}{\rightharpoonup}$ are the weak and weak-$*$ convergence, respectively. With the aid of the above
strong and weak convergence, one can take the limit, as $n\rightarrow\infty$, to show that $(v,T)$ satisfies the boundary
conditions (\ref{73})--(\ref{75}) in the sense of trace, satisfies equations (\ref{16})--(\ref{17}) a.e.\,in
$\Omega\times(0,\infty)$, and fulfills the initial condition (\ref{8}).
By the weakly lower semi-continuity of norms, one can verify from the above convergence that $(v,T)$ satisfies all the
regularities stated in Definition \ref{def} except that the regularities
$v\in C([0,\mathfrak T]; \mathcal H)$ and $T\in C([0,\mathfrak T]; L^2(\Omega))$ are replaced by the weaker ones: $v\in
C((0,\mathfrak T]; \mathcal H)\cap C([0,\mathfrak T]; \mathcal V^*)$ and $T\in C((0,\mathfrak T]; L^2(\Omega))\cap
C([0,\mathfrak T]; (H^1(\Omega))^*)$.
Therefore, in order to complete the proof of the existence part of Theorem \ref{thm1}, it remains to verify that
$$
(v,T)\rightarrow(v_0,T_0)\quad\mbox{ in }L^2(\Omega),\mbox{ as } t\rightarrow0.
$$

Note that $\mathcal V\hookrightarrow L^2(\Omega)$ and $H^1(\Omega)\hookrightarrow L^2(\Omega)$, one has 
$$
L^2(\Omega)=(L^2(\Omega))^*\hookrightarrow\mathcal V^*\quad\mbox{and}\quad 
L^2(\Omega)\hookrightarrow (H^1(\Omega))^*.
$$ 
Thanks to
these, recalling $v\in C([0,\mathfrak T]; \mathcal V^*)$, $T\in C([0,\mathfrak T]; (H^1(\Omega))^*)$, and $(v,T)\in L^\infty(0,\mathfrak T; L^2(\Omega))$, it follows from Lemma \ref{lemmaADLI} that
$(v,T)\in C_w([0,\mathfrak T]; L^2(\Omega))$, where $C_w$ means the weak continuity. With the aid of this, it suffices to show that
$$
\varlimsup_{t\rightarrow0}(\|v\|_{L^2}^2+\|T\|_{L^2}^2)(t)\leq \|v_0\|_{L^2}^2+\|T_0\|_{L^2}^2.
$$
This can be verified as follows: Proposition \ref{pro3.1} implies
$$
(\|v_n\|_{L^2}^2+\|T_n\|_{L^2}^2)(t)\leq e^{Ct}(\|v_{0n}\|_{L^2}^2+\|T_{0n}\|_{L^2}^2),\quad\forall t>0,
$$
from which, recalling that $(v_n,T_n)\rightarrow (v,T)$ in $C([\delta,\mathfrak{T}];L^2(\Omega))$ for any $\delta\in(0,\mathfrak T)$ and $(v_{0n}, T_{0n})\rightarrow(v_0,T_0)$ in $L^2(\Omega)$, one can take the limit $n\rightarrow\infty$ in the above to get $(\|v\|_{L^2}^2+\|T\|_{L^2}^2)(t)  \leq e^{Ct}(\|v_0\|_{L^2}^2+\|T_0\|_{L^2}^2)$, leading to the desired
inequality.

\textbf{(ii) Uniqueness and continuous dependence.}
Let $(v_1,T_1)$ and $(v_2,T_2)$ be two $z$-weak solutions to system \eqref{16}--\eqref{17}, subject to \eqref{73}--\eqref{75}, with initial data $(v_{10},T_{10})$ and $(v_{20},T_{20})$, respectively. Denote $\omega:=v_1-v_2$, $\theta :=T_1-T_2$, $q:=p_{s1}-p_{s2}$, $\omega_0:=v_{10}-v_{20}$ and $\theta_0:=T_{10}-T_{20}$. Then,
\begin{align}
\partial_t\omega&+(v_2\cdot\nabla_H) \omega-\left(\int_{-h}^z\nabla_H\cdot v_2d\xi\right)\partial_z\omega+(\omega\cdot\nabla_H) v_1-\left(\int_{-h}^z\nabla_H\cdot \omega d\xi\right)\partial_zv_1
\nonumber\\
&+\nabla_Hq-\nabla_H\int_{-h}^z\theta d\xi+fk\times \omega-\frac{1}{Re_1}\Delta_H \omega-\frac{1}{Re_2}\partial_z^2 \omega=0,\label{DIFF1}
\end{align}
and
\begin{equation}\label{DIFF2}
\begin{aligned}
\partial_t \theta+v_2\cdot\nabla_H \theta&-\left(\int_{-h}^z\nabla_H\cdot v_2d\xi\right)\partial_z\theta+\omega\cdot\nabla_H T_1\\
&-\left(\int_{-h}^z\nabla_H\cdot \omega d\xi\right)\partial_z T_1-\frac{1}{R_T}\Delta_H \theta=0,
\end{aligned}
\end{equation}
a.e.\,in $\Omega\times(0,\infty)$, and the following boundary conditions hold:
\begin{equation*}
\begin{aligned}
&\textrm{on}~\Gamma_s:~\partial_n \theta=-\alpha_T\theta,~\omega=0,\\
&\textrm{on}~\Gamma_u:~\partial_z \omega=0,\\
&\textrm{on}~\Gamma_b:~\partial_z \omega=0.
\end{aligned}
\end{equation*}
Multiplying (\ref{DIFF1}) and (\ref{DIFF2}) with $\omega$ and $\theta$, respectively, it follows from Lemma \ref{lemma1} and integrating by parts that
\begin{align*}
& \frac{1}{2}\frac{d}{dt}\|\omega\|_{L^2}^2+\frac{1}{Re_1}\|\nabla_H \omega\|_{L^2}^2+\frac{1}{Re_2}\|\partial_z\omega\|_{L^2}^2\\
=&-\big((\omega\cdot\nabla_H)v_1,\omega\big)_{L^2}+\left(\int_{-h}^z\nabla_H\cdot \omega d\xi\partial_z v_1,\omega\right)_{L^2}+\left(\nabla_H\int_{-h}^z\theta d\xi,\omega\right)_{L^2}\\
\leq& C\int_M\int_{-h}^0\Big(|\nabla_H\partial_zv_1|+|\nabla_Hv_1|\Big)dz\int_{-h}^0|\omega|^2dzdxdy\\
& +\int_M\int_{-h}^0|\nabla_H\omega|dz\int_{-h}^0|\partial_zv_1||\omega|dzdxdy
+\int_M\int_{-h}^0|\nabla_H\theta|dz\int_{-h}^0|\omega|dzdxdy\\
\leq& C\|(\nabla_H v_1,\partial_z\nabla_H v_1)\|_{L^2}\big)\|\omega\|_{L^2}\|(\omega,\nabla_H\omega)\|_{L^2}+C\|\nabla_H\theta\|_{L^2}\|\omega\|_{L^2}\\
& +C\|\nabla_H\omega\|_{L^2}\|\partial_z v_1\|_{L^2}^{\frac{1}{2}}\|(\partial_z v_1,\partial_z\nabla_Hv_1)\|_{L^2}^{\frac{1}{2}}\|\omega\|_{L^2}^{\frac{1}{2}}\|(\omega,\nabla_H\omega)\|_{L^2}
^{\frac{1}{2}}\\
\leq& C\big(1+\|(\nabla_H v_1,\partial_z\nabla_Hv_1)\|_{L^2}^2+\|\partial_z v_1\|_{L^2}^4+\|\partial_z v_1\|_{L^2}^2\|\partial_z\nabla_Hv_1\|_{L^2}^2\big)\|\omega\|_{L^2}^2\\
& +\frac{1}{4Re_1}\|\nabla_H \omega\|_{L^2}^2+\frac{1}{4R_T}\|\nabla_H\theta\|_{L^2}^2,
\end{align*}
where we have used $|f|\leq \int_{-h}^0|\partial_zf|dz+\frac{1}{h}\int_{-h}^0|f|dz$. Similarly,
\begin{align*}
& \frac{1}{2}\frac{d}{dt}\|\theta\|_{L^2}^2+\frac{1}{R_T}\|\nabla_H \theta\|_{L^2}^2+\frac{\alpha_T}{R_T}\|\theta\|_{L^2(\Gamma_s)}^2\\
=&-(\omega\cdot\nabla_HT_1,\theta)_{L^2}+\left(\int_{-h}^z\nabla_H\cdot \omega d\xi\partial_z T_1,\theta\right)_{L^2}\\
\leq& C\int_M\int_{-h}^0\Big(|\nabla_HT_1|+|\nabla_H\partial_zT_1|\Big)dz\int_{-h}^0|\omega||\theta|dzdxdy\\
 &+C\int_M\int_{-h}^0|\nabla_H\omega|dz\int_{-h}^0|\partial_zT_1||\theta|dzdxdy\\
\leq &C\|(\nabla_HT_1,\nabla_H\partial_zT_1)\|_{L^2} \|\omega\|_{L^2}^{\frac{1}{2}}\|(\omega,\nabla_H\omega)\|_{L^2}^{\frac{1}{2}}  \|\theta\|_{L^2}^{\frac{1}{2}}\|(\theta,\nabla_H\theta)\|_{L^2}^{\frac{1}{2}}\\
&+C\|\nabla_H\omega\|_{L^2}\|\partial_zT_1\|_{L^2}^{\frac{1}{2}}\|(\partial_zT_1,\nabla_H\partial_zT_1)\|_{L^2}^{\frac{1}{2}}
\|\theta\|_{L^2}^{\frac{1}{2}}\|(\theta,\nabla_H\theta)\|_{L^2}^{\frac{1}{2}}\\
\leq& C\big(1+\|(\nabla_HT_1,\partial_z\nabla_HT_1)\|_{L^2}^2+\|\partial_zT_1\|_{L^2}^4+\|\partial_zT_1\|_{L^2}^2\|\partial_z\nabla_HT_1\|_{L^2}^2\big)\\
& \times\|(\theta,\omega)\|_{L^2}^2+\frac{1}{4Re_1}\|\nabla_H\omega\|_{L^2}^2+\frac{1}{4R_T}\|\nabla_H\theta\|_{L^2}^2.
\end{align*}
Combining the above estimates, one has
\begin{equation*}
\begin{aligned}
\frac{d}{dt}\Big(\|\omega\|_{L^2}^2+\|\theta\|_{L^2}^2\Big)&+\frac{1}{Re_1}\|\nabla_H\omega\|_{L^2}^2+\frac{1}{Re_2}\|\partial_z\omega\|_{L^2}^2
+\frac{1}{R_T}\|\nabla_H\theta\|_{L^2}^2\\
\leq& C\big(G_1(t)+G_2(t)\big)\big(\|\omega\|_{L^2}^2+\|\theta\|_{L^2}^2\big),
\end{aligned}
\end{equation*}
where
\begin{equation*}
\begin{aligned}
&G_1(t):=1+\|(\nabla_HT_1,\partial_z\nabla_HT_1)\|_{L^2}^2+\|\partial_zT_1\|_{L^2}^4+\|\partial_zT_1\|_{L^2}^2\|\partial_z\nabla_HT_1\|_{L^2}^2,\\
&G_2(t):=1+\|(\nabla_H v_1,\partial_z\nabla_Hv_1)\|_{L^2}^2+\|\partial_z v_1\|_{L^2}^4+\|\partial_z v_1\|_{L^2}^2\|\partial_z\nabla_Hv_1\|_{L^2}^2.
\end{aligned}
\end{equation*}
Note that the regularities of $(v_i,T_i)$ imply that $G_1,G_2\in L^1((0,\mathfrak T))$ for any $\mathfrak T>0$.
Then, for any $0<s<\mathfrak{T}<\infty$, applying the Gr\"onwall inequality on $t\in[s,\mathfrak{T}]$, one can obtain
\begin{eqnarray*}
\sup_{s\leq t\leq \mathfrak{T}}\|(\omega,\theta)\|_{L^2}^2(t)+\int_s^{\mathfrak{T}} \|(\nabla\omega,\nabla_H\theta)\|_{L^2}^2
 dt
\leq e^{C\int_s^{\mathfrak{T}}\big(G_1(t)+G_2(t)\big)dt}\|(\omega,\theta)\|_{L^2}^2(s),
\end{eqnarray*}
from which, letting $s\rightarrow0$, one gets
\begin{eqnarray*}
\sup_{0\leq t\leq \mathfrak{T}}\|(\omega,\theta)\|_{L^2}^2(t)+\int_0^{\mathfrak{T}} \|(\nabla\omega,\nabla_H\theta)\|_{L^2}^2
 dt
\leq e^{C\int_0^{\mathfrak{T}}\big(G_1(t)+G_2(t)\big)dt}\|(\omega_0,\theta_0)\|_{L^2}^2.
\end{eqnarray*}
This leads to the uniqueness and continuous dependence on the initial datum.
\end{proof}

\section{Appendix}
It is shown in this appendix that one can homogenize the nonhomogeneous boundary conditions at the cost of having some
extra linear terms in the resulting equations. We only give the details about the
system with full dissipation, that is, system \eqref{19}--\eqref{21}; however, calculations for the system (\ref{16})--(\ref{17}) are similar. The following boundary conditions are complemented to system (\ref{19})--(\ref{21}):
\begin{equation*}
\begin{aligned}
&\textrm{on}~\Gamma_s:~\partial_n T=\alpha_T(T_{s}-T),~v=0,\\
&\textrm{on}~\Gamma_u:~\partial_z v=-\alpha_v\tau(x,y),~\partial_z T=0,\\
&\textrm{on}~\Gamma_b:~\partial_z v=0,~\partial_z T=0.
\end{aligned}
\end{equation*}

Let $V=v+\frac{\alpha_v}{h}\left(\frac{(z+h)^2}{2}-\frac{h^2}{6}\right)\tau(x,y)$. Then, $\partial_z V=\partial_z v+\frac{\alpha_v}{h}(z+h)\tau$ and, thus, $\partial_z V=0$ on $\Gamma_b$ and $\Gamma_u$. By \eqref{12}, one can check that

\begin{equation*}
V\Big|_{\Gamma_s}=v|_{\Gamma_s}+\frac{\alpha_v}{h}\left(\frac{(z+h)^2}{2}-\frac{h^2}{6}\right)\tau|_{\Gamma_s}=0\qquad\textrm{on}~\Gamma_s.
\end{equation*}
Obviously, one has
\begin{equation*}
\nabla_H\cdot\int_{-h}^0Vdz=0.
\end{equation*}

Let $T^*$ be the unique solution to
\begin{equation*}
\begin{aligned}
&\partial_tT^*-\Delta T^*=0\qquad\qquad\ \ \ \textrm{in}~\Omega,\\
&\partial_n T^*+\alpha_T T^*=\alpha_T T_{s}~\qquad \textrm{on}~\Gamma_s,\\
&\partial_z T^*=0~\qquad\qquad\qquad\textrm{on}~\Gamma_u~\textrm{and}~\Gamma_b,\\
&T|_{t=0}=0.
\end{aligned}
\end{equation*}
Existence, uniqueness and regularity of such $T^*$ has been established in \cite{Hit1}.
Set $\mathcal{T}:=T-T^*$. Then, it is clear that
\begin{equation*}
\partial_n \mathcal{T}|_{\Gamma_s}=\partial_n T|_{\Gamma_s}-\partial_n T^*|_{\Gamma_s}=-\alpha_T \mathcal{T}|_{\Gamma_s} 
\end{equation*}
and
\begin{equation*}
\partial_z\mathcal{T}=0\quad \textrm{on}~\Gamma_u\cup\Gamma_b.
\end{equation*}

Substituting $v=V-\frac{\alpha_v}{h}\left(\frac{(z+h)^2}{2}-\frac{h^2}{6}\right)\tau$ and $T=T^*+\mathcal{T}$ into \eqref{19}--\eqref{21} yields
\begin{equation*}
\begin{aligned}
\partial_t V-\frac{1}{Re_1}\Delta_H V&-\frac{1}{Re_2}\partial_z^2 V+\big(V\cdot\nabla_H\big) V-\left(\int_{-h}^z\nabla_H\cdot Vd\xi\right)\partial_z V\\
&+\nabla_H p_s-\nabla_H\int_{-h}^z\mathcal{T}d\xi+fk\times V+a_\tau(V)=F_\tau,
\end{aligned}
\end{equation*}
\begin{equation*}
\nabla_H\cdot\int_{-h}^0Vdz=0,
\end{equation*}
\begin{equation*}
\partial_t\mathcal{T}-\frac{1}{R_T}\Delta_H\mathcal{T}-\varepsilon\partial_z^2\mathcal{T}+V\cdot\nabla_H\mathcal{T}-\left(\int_{-h}^z\nabla_H\cdot Vd\xi\right)\partial_z \mathcal{T}+b(V,\mathcal{T})=G_\tau,
\end{equation*}
where
\begin{align*}
a_\tau(V)=&-\frac{\alpha_v}{h}\left(\frac{(z+h)^2}{2}-\frac{h^2}{6}\right)\Big[\big(V\cdot\nabla_H\big)\tau+
\big(\tau\cdot\nabla_H\big) V\Big]\\
&+\frac{\alpha_v}{6h}\Big((z+h)^3-h^2(z+h)\Big)(\nabla_H\cdot\tau)\partial_z V +\left(\int_{-h}^z\nabla_H\cdot Vd\xi\right)\frac{\alpha_v(z+h)}{h}\tau,\\
b(V,\mathcal{T})=&-\left(\int_{-h}^z\nabla_H\cdot Vd\xi\right)\partial_z T^*-\frac{\alpha_v}{h}\left(\frac{(z+h)^2}{2}-\frac{h^2}{6}\right)(\tau\cdot\nabla_H)\mathcal{T}\\
&+\frac{\alpha_v}{6h}\Big((z+h)^3-h^2(z+h)\Big)(\nabla_H\cdot\tau)\partial_z\mathcal{T}+V\cdot\nabla_H T^*,
\end{align*}
and
\begin{align*}
F_\tau=&\frac{\alpha_v}{h}\left(\frac{(z+h)^2}{2}-\frac{h^2}{6}\right)(fk\times\tau+\partial_t\tau)
-\frac{\alpha_v^2}{h^2}\left(\frac{(z+h)^2}{2}-\frac{h^2}{6}\right)^2(\tau\cdot\nabla_H)\tau\\
&-\frac{\alpha_v}{Re_1h}\left(\frac{(z+h)^2}{2}-\frac{h^2}{6}\right)\Delta_H\tau -\frac{\alpha_v\tau}{Re_2h}+\nabla_H\int_{-h}^zT^*d\xi\\
&+\frac{\alpha_v^2}{6h^2}\Big((z+h)^4-h^2(z+h)^2\Big)(\nabla_H\cdot\tau)\tau,\\
G_\tau=&\frac{\alpha_v}{h}\left(\frac{(z+h)^2}{2}-\frac{h^2}{6}\right)\tau\cdot\nabla_H T^*-\frac{\alpha_v}{6h}\Big((z+h)^3-h^2(z+h)\Big)(\nabla_H\cdot\tau)\partial_zT^*\\
&-\partial_t T^*+\frac{1}{R_T}\Delta_H T^*+\varepsilon\partial_z^2T^*.
\end{align*}

The initial and boundary conditions read as
\begin{equation*}
\big(V,\mathcal{T}\big)\Big|_{t=0}=\left(v_0+\frac{\alpha_v}{h}\left(\frac{(z+h)^2}{2}-\frac{h^2}{6}\right)
\tau(x,y,0),T_0\right) 
\end{equation*}
and
\begin{equation*}
\begin{aligned}
\textrm{on}~\Gamma_s:\ \  &V=0,~\partial_n\mathcal{T}+\alpha_T\mathcal{T}=0,
\end{aligned}
\end{equation*}
\begin{equation*}
\textrm{on}~\Gamma_u\cap\Gamma_b:\  \partial_z V=0,~\partial_z\mathcal{T}=0.
\end{equation*}

Following the same arguments as in this paper, one can check that there is no difficulty to deal with the linear terms $a_\tau(V)$, $b(V,\mathcal{T})$, $F_\tau$, and $G_\tau$, if $\tau$ and $T_s$ are smooth enough and, as a result, Theorem \ref{thm1} still holds true for the corresponding nonhomogeneous boundary conditions.

\smallskip
{\bf Acknowledgment.}
This work was supported in part by the National
Natural Science Foundation of China (11971009, 11871005, and 11771156) and by
the Guangdong Basic and Applied Basic Research Foundation (2019A1515011621,
2020B1515310005, 2020B1515310002, and 2021A1515010247).

\bigskip

\end{document}